\documentclass[1pt]{article}
\usepackage{amsmath, amssymb, amsthm, amsfonts, cases}
\usepackage{mathrsfs}
\usepackage{mathtools}
\usepackage{url}
\usepackage{authblk}
\usepackage[usenames]{color}
\usepackage{geometry}
\theoremstyle{plain}
\newtheorem{thm}{Theorem}[section]
\newtheorem{cor}[thm]{Corollary}
\newtheorem{pro}[thm]{Problem}
\newtheorem{lem}[thm]{Lemma}

\theoremstyle{definition}
\newtheorem{defn}{Definition}[section]
\newtheorem{ass}{Assumption}[section]

\renewcommand{\d}{\mathrm{d}}
\newcommand{\eps}{\varepsilon}

\makeatletter\@addtoreset{equation}{section} \makeatother

\begin{document}

\title{Forward and Backward Mean-Field Stochastic Partial Differential Equation
and Optimal Control
\thanks{This work was supported by the Natural Science Foundation of Zhejiang Province
for Distinguished Young Scholar  (No.LR15A010001),  and the National Natural
Science Foundation of China (No.11471079, 11301177) }}

\date{}

\author[a]{Maoning Tang}
\author[a]{Qingxin Meng\footnote{Corresponding author.
\authorcr
\indent E-mail address: mqx@zjhu.edu.cn(Q. Meng),}}

\affil[a]{\small{Department of Mathematical Sciences, Huzhou University, Zhejiang 313000, China}}

\maketitle

\begin{abstract}
\noindent

This paper is mainly concerned with
 the solutions  to   both forward and backward mean-field stochastic partial differential equation and  the corresponding optimal
control problem for mean-field stochastic
partial differential equation. We first prove
the continuous dependence theorems of
forward and backward mean-field stochastic partial differential equations and  show the existence and uniqueness of solutions to them. Then we establish necessary
and sufficient optimality  conditions of
the control problem in
the form of Pontryagin's maximum principles. To illustrate the theoretical results,
we apply stochastic maximum principles to study an example, an infinite-dimensional linear-quadratic control problem of mean-field type. Further an application  to a Cauchy problem for a controlled stochastic linear PDE of mean-field type  are studied.
\end{abstract}
\section{Introduction}

 In recent years, due to many practical and theory applications, in the finite dimensional cases, the  stochastic differential
equation of mean-field type,
also called  mean-field stochastic
differential equation (MF-SDE), and
the corresponding optimal control
problem and financial applications
  has been studied expensively.
  For more details on these
  topics,  the interested
  reader is referred to Kac
(1956),  McKean (1966),
  Andersson and Djehiche (2011), Buckdahn, Djehiche and Li (2011),
Li (2012), Meyer-Brandis, ${\O}$ksendal and Zhou (2012),  Elliott, Li and  Ni (2013), Shen
and Siu (2013), Shen, Meng and Shi (2014),
 Wang, Zhang  and Zhang (2014), Meng and Shen (2015),
Yong(2013), Du, Huang and Qin (2013)     and therein.
On the other hand, intuitively speaking, the adjoint equation of a controlled state
process driven by the MF-SDE is a mean-field backward
stochastic differential equation (MF-BSDE).
So it is not until Buckdahn et al. (2009a,
2009b) established the theory of the MF-BSDEs that the optimal control problem of mean-field type has become a popular topic
where the adjoint equation associated with
the stochastic maximum principle  is a  MF-BSDE.

The purpose of this  paper is
to extend  finite dimensional MF-SDE and MF-BSDE and the corresponding optimal control
problem to infinite dimensional case,
i.e., to mean-field stochastic partial differential equations (MF-SPDE) and backward mean-field stochastic partial differential equations (MF-BSPDE). We will establish the basic theorem  of  MF-SPDE and
MF-BSPDE and  the basic optimal control theorem  for MF-SPDE. Precisely speaking, by It\^o's formula in the Gelfand triple and under some proper assumptions,  we firstly  prove continuous dependence property of the solution to  both MF-SPDE and MF-BSPDE on the parameter. Then the existence and uniqueness of solutions to MF-SPDE and MF-BSPDE are proved by
the continuous dependence theorem and
the classic parameter extension approach.
The second main result established  in  this paper   is  the corresponding sufficient and necessary stochastic maximum principle for  the optimal control problem of MF-BSPDE which  are obtained by establishing an convex variation formula under the  convexity assumption of the control domain.
Finally, to illustrate our
results, we apply the stochastic maximum principles to a  mean-field linear-quadratic (LQ) control problem of BF-SPDE. Using the necessary and sufficient maximum principles,
the optimal control strategy is given explicitly in a dual representation. As an application,
a LQ problem for a concrete cauchy problem of
 controlled  mean-field stochastic partial equation is solved.

 The rest of this paper is organized as follows. Section 2 gives notations and framework.
 In section 3, we prove  the continuous
 dependence theory and the existence
 and uniqueness  of solutions to
 MF-SPDE in the abstract form.
 In section 4, we prove  the continuous
 dependence theory and the existence
 and uniqueness  of solutions to
 MF-BSPDE in the abstract form.
In Section 5, the optimal control problem
of MF-SPDE is studied in detail where we establish the stochastic sufficient and necessary maximum principles under convex
control domain assumption. Sections  6 applys the stochastic maximum principles to solve  linear -quadratic optimal control problems
 of MF-SPDE. The final section concludes the paper.

\section{Notations}

Denote by $\cal T $ the
fixed time duration $[0, T].$
Let $(\Omega, {\mathscr  F},  \mathbb F, {\mathbb P})$
be a complete
probability space
on which one-dimensional real-valued
Brownian  motion  $\{ W(t), 0 \leq t \leq T \}$
is defined with
${\mathbb F}
\triangleq \{ {\mathscr F}_t, {0\leq t\leq T} \}$  being its
natural filtration  augmented by all the $\mathbb P$-null sets.
Denote by  ${\mathbb E} [\cdot]$
the expectation with respect to  the
probability ${\mathbb P}$. We denote by $\mathscr P $ the predictable $\sigma-$
algebra associated with ${\mathbb F}.$
For any topological space $\Lambda, $  we denote by  ${\mathscr B} (\Lambda)$  its Borel $\sigma$-algebra. Let $X$ be any Hilbert space in  which the norm is denoted by$\| \cdot \|_X$ .
 Next we introduce the following spaces:
\begin{enumerate}
\item[$\bullet$] ${\cal M}_{\mathscr F}^2 (0, T; X)$: the space  of all $X$-valued ${\mathscr F}$-adapted processes
$f \triangleq \{ f (t, \omega), (t,\omega) \in [0, T] \times \Omega \}$ endowed with the norm
$\| f \|_{{\cal M}_{\mathscr F}^2 (0, T; X)} \triangleq \sqrt {
{\mathbb E} \big [ \int_0^T \| f(t) \|_X^2 d t \big ] } < \infty$;
\item[$\bullet$] ${\cal S}_{\mathscr F}^2 (0, T; X)$: the space of all $X$-valued ${\mathscr F}$-adapted c\`adl\`ag processes
$f \triangleq \{ f (t, \omega), (t,\omega) \in [0, T] \times \Omega \}$ endowed with the norm
$\| f \|_{{\cal S}_{\mathscr F}^2 (0, T; X)} \triangleq \sqrt { {\mathbb E} \big [
\sup_{0 \leq t \leq T} \| f(t) \|_X^2 \big ] } < \infty$;
\item[$\bullet$] ${\cal L}^p (\Omega, \mathscr F, \mathbb P; X)$: the space of all $X$-valued ${\mathscr F}$-measurable random variables
$\xi$ endowed with the norm $\| \xi \|_{{\cal L}^p (\Omega, \mathscr F, \mathbb P; X)} \triangleq \sqrt {{\mathbb E} \big [ \| \xi \|_X^p \big ] } < \infty,$ where $p\geq 1$ are given real number.
\end{enumerate}

\section {Mean-Field Stochastic Partial  Differential  Equation}
This section is devoted the study of a new
 type stochastic partial differential equation
 with abstract form, the so called MF-SPDE.
  Let $(\bar \Omega, \bar {\mathscr F}, \bar {\mathbb P})=(\Omega\times \Omega,
\mathscr F\times \mathscr F, \mathbb P\times
\mathbb P)$ be the product of $(\Omega,
\mathscr F, \mathbb P)$ with itself.   We endow this product space with the
filtration
 $\{{ \mathscr{\bar F}}_t\}_{0\leq t\leq T}
 =\{{\mathscr{ F}}_t\times {\mathscr{ F}}_t\}_{0\leq t\leq T}.
 $ By $\mathscr {\bar P}$ we denote
  the product $\mathscr P\times \mathscr P.$
  Let $\bar {\mathbb E}$  denotes
  the expectation with respect to
  the product probability space
  $\bar \Omega.$ Denote by  ${\cal M}_{\bar {\mathscr F}}^2 (0, T; X)$ the set of all $X$-valued $\bar {\mathscr F}$-adapted processes
$f \triangleq \{ f (t, \omega',\omega), (t,\omega', \omega) \in [0, T] \times \bar \Omega \}$ such that
$\| f \|_{{\cal M}_{\bar{\mathscr F}}^2 (0, T; X)} \triangleq \sqrt {
\bar {\mathbb E} \big [ \int_0^T \| f(t) \|_X^2 d t \big ] } < \infty.$ For $p\geq 1,$ a random variable $\xi \in L^p(\Omega, \mathscr F, P; X)$ originally defined
 on $\Omega$ can be extended
 canonically to $\bar \Omega$: $\xi'(\omega',\omega)=\xi(\omega'),(\omega', \omega)\in \bar \Omega.$
 For any $\theta\in  {\cal L}^1(\bar\Omega, \bar {\mathscr F}, \bar P; X),$
 the variable $\theta(\cdot, \omega):\Omega\longrightarrow  X$
belongs to ${\cal L}^p(\Omega,  {\mathscr F},
\mathbb P; X),$ $\mathbb P(d\omega)-$a.s., we
denote its expectation by
$$\mathbb E'[\theta(\cdot, \omega)]=\int_\Omega \theta (\omega', \omega))\mathbb P(d\omega').$$
Notice that $\mathbb E'[\theta]=
\mathbb E'[\theta(\cdot, \omega)]\in
L^p(\Omega, {\mathscr F}, \mathbb P; X),$ and $$\bar {\mathbb
E}[\theta](=\int_{\bar \Omega}\theta d\bar {\mathbb P}=\int_{\Omega}\mathbb
E'[\theta(\cdot, \omega)]\mathbb Pd(\omega))=\mathbb E[\mathbb
E'[\theta]].$$
Let  $$V \subset H = H^*\subset V^*$$ be  Gelfand triple, i.e., $(H, (\cdot, \cdot)_H)$
 is a separable  Hilbert spaces and $V$ is
  a reflextive Banach space
  such that  $H$ is identified with its dual space
  $H^*$
by the Riesz isomorphism and  $V$
is densely embedded in $H$. We denote by $\left < \cdot,\cdot \right >$ the duality product between
$V$ and $V^{*}$. Moreover, we  denote  by  ${\mathscr L} (V, V^*)$  the set of all bounded
linear operators from  $V$ into $V^*$.
In  the
Gelfand triple $(V, H, V^*),$
consider  the following operators
\begin{eqnarray} \label{eq:3.1}
  A&=&A(t,\omega) : [ 0, T ] \times \nonumber
\Omega \rightarrow {\mathscr L} ( V, V^* ),\nonumber
\\b&=&b(t,\omega', \omega, x', x) : [0, T] \times
\bar \Omega \times H \times H \rightarrow H,
\nonumber
\\g&=&g(t,\omega', \omega, x', x) : [0, T] \times
\bar \Omega \times H \times H \rightarrow H,
\end{eqnarray}
which satisfy  the following  standard assumptions:

\begin{ass}\label{ass:1.1}
Suppose that there exist constant
$\alpha>0, \lambda, $ and $C$
such that the following  conditions
holds for all $x,x', \bar x, \bar x'$
and  a.e. $( t, \omega',\omega ) \in [ 0, T ] \times \bar\Omega$
\begin{enumerate}
\item[(i)](Measurability) The operator $A$ is ${\mathscr P} / {\mathscr B} ( {\mathscr L} ( V, V^* ) )$
measurable; $b$ and $g $ are $\bar {\mathscr P} \otimes
{\mathscr B} (H) \otimes {\mathscr B} (H) / {\mathscr B} (H)$ measurable;
\item[(ii)] (Integrality) $b ( \cdot, 0, 0 ), g ( \cdot, 0, 0 )\in {\cal M}_{\bar {\mathscr F}}^2 ( 0, T; H )$;
\item[(ii)] (Coercivity)
\begin{eqnarray} \label{eq:2.22}
\left < A (t) x, x \right > + \lambda \| x \|^2_H \geq \alpha \| x \|^2_V \  ;
\end{eqnarray}
\item[(iv)](Boundedness)
\begin{eqnarray}
\sup_{( t, \omega ) \in [0, T] \times \Omega} \| A ( t,\omega ) \|_{{\mathscr L} ( V, V^* )} \leq C \ ;
\end{eqnarray}
\item[(iv)] (Lipschitz Continuity)
\begin{eqnarray}\label{eq:2.40}
\| b ( t, x', x) - b ( t, {\bar x'}, {\bar x} ) \|_H
+ \| g ( t, x', x ) - g ( t, {\bar x'}, {\bar x} ) \|_H
\leq C [ \| x - {\bar x} \|_H + \| x' - {\bar x}' \|_H ] \ .
\end{eqnarray}
\end{enumerate}
\end{ass}

Using  the above  notation, in   the
Gelfand triple $(V, H, V^*),$ we consider
the MF-SPDE in the following abstract form with the
coefficients $(A, b,g)$ defined by \eqref{eq:3.1} and the initial value $x\in H$ :
\begin{eqnarray}\label{eq:2.1}
\left
\{\begin{aligned}
d X(t) = & \ \big\{ - A (t) X (t)+ \mathbb E'[b ( t,  X' (t), X ( t  ) ]\big\} d t
+ \mathbb E'[ g( t, x' (t), X (t) )] d W(t) \ , \quad t \in [ 0, T ] \ , \\
X (t) = & \ x \in  H \ ,
\end{aligned}
\right.
\end{eqnarray}
where we have used the following notation defined by
\begin{eqnarray}
   \mathbb E'[b ( t,  X' (t), X ( t) ]=
   \int_\Omega b(t, \omega', \omega, X(t,\omega'), X(t,\omega)) P(d \omega')
\end{eqnarray}
and

\begin{eqnarray}
   \mathbb E'[g ( t,  X' (t), X ( t) ]=
   \int_\Omega g(t, \omega', \omega, X(t,\omega'), X(t,\omega)) P(d \omega')
\end{eqnarray}

Now we give  the definition of  the solutions to
MF-SPDE \eqref{eq:2.1}.
\begin{defn}\label{defn:c1}
An $V$-valued, ${\mathbb F}$-adapted process $X(\cdot)$ is said to be  a solution to
MF-SPDE \eqref{eq:2.1}, if $X (\cdot) \in {\cal M}_{\mathscr F}^2 ( 0, T; V )$, such that for a.e. $( t, \omega ) \in [ 0, T ] \times \Omega$ and  every $\phi \in V,$ we have
\begin{eqnarray}
\begin{split}
( X(t), \phi )_H =& \ ( x, \phi )_H - \int_0^t \left < A (s) X (s), \phi \right > d s
+\int_0^t \big( \mathbb E'[b ( s, X' (s), X ( s) )], \phi \big)_H d s \\
& +\int_0^t \big( \mathbb E'[g ( s, X' (s), X ( s  ) )], \phi \big)_H d W (s) \ , \quad t \in [ 0, T ] \ ,
\end{split}
\end{eqnarray}
or alternatively, in the sense of $V^*$, $X (\cdot)$ have  the following It\^o form:
\begin{eqnarray}
\begin{split}
 X(t) =& \ x - \int_0^t A (s) X (s) d s + \int_0^t \mathbb E'[ b ( s, X'(s), X(s) )] d s+\int_0^t  \mathbb E'[g ( s, X' (s), X ( s ) )] d W (s) \ , \quad t \in [ 0, T ] \ .
\end{split}
\end{eqnarray}
\end{defn}
The following result is
the continuous dependence theorem of the solution
 to the MF-SPDE \eqref{eq:2.1} on the coefficients $(A,b,g)$ and the
initial value $x$ which  is also called a priori
estimate for the solution.

\begin{thm}[{\bf Continuous dependence theorem of MF-SPDE}] \label{lem:1.2}
    Suppose that $X (\cdot)$ is a solution to MF-SPDE \eqref{eq:2.1}
    with the initial value $x$ and the coefficients $(A,b,g)$ satisfying Assumptions \ref{ass:1.1}.
   Then we have the following estimate:
\begin{eqnarray}\label{eq:3.10}
\begin{split}
 {\mathbb E} \bigg [ \sup_{0 \leq t \leq T} \|X(t) \|_H^2 \bigg
]
+  {\mathbb E} \bigg [ \int_0^T \| X (t) \|_V^2 d t \bigg ]
 \leq K \bigg \{ \bar {\mathbb E} [ \| x\|^2_H ] + \bar {\mathbb E}
\bigg [ \int_0^T \| b ( t, 0, 0 ) \|_H^2 d t \bigg ] + \bar {\mathbb E}
\bigg [ \int_0^T \| g ( t, 0, 0 ) \|_H^2 d t \bigg ]\bigg \} \ ,
\end{split}
\end{eqnarray}
where $K$ is a positive constant
which only depend on the constants  $C, T, \alpha$ and
$\lambda.$
 Further, suppose that  ${\bar X} (\cdot)$ is  the
solution to MF-SPDE  \eqref{eq:2.1}
 with the initial value $\bar x$ and the coefficients  $( A, {\bar b}, {\bar g} )$ satisfying Assumption \ref{ass:1.1}.
   Then we have
\begin{eqnarray}\label{eq:3.111}
\begin{split}
& {\mathbb E} \bigg [ \sup_{0 \leq t \leq T} \| X (t) - {\bar X} (t) \|_H^2 \bigg ]
+ {\mathbb E} \bigg [ \int_0^T \| X (t) - {\bar X} (t) \|_V^2 d t \bigg ]  \\
& \leq K \bigg \{{\bar {\mathbb  E}} \bigg [ \int_0^T \| b ( t, {\bar X}' (t), {\bar X} ( t) )
- {\bar b}( t, {\bar X'} (t), {\bar X} ( t ) ) \|_H^2 d t \bigg ] + {\bar {\mathbb E}} \bigg [ \int_0^T \| g ( t, {\bar X'} (t),
{\bar X} ( t ) ) - {\bar g}( t, {\bar X'} (t), {\bar X} ( t ) ) \|_H^2 d t \bigg ]\bigg\} \ .
\end{split}
\end{eqnarray}
\end{thm}

\begin{proof}
 It suffices to prove \eqref{eq:3.111}
 since the estimate \eqref{eq:3.10} can be obtained as a direct consequence of
 \eqref{eq:3.111} by taking the coefficient
$( A, {\bar b}, {\bar g} ) = (A, 0, 0)$
with which the solution to MF-SPDE \eqref{eq:2.1} is ${\bar x} (\cdot) = 0.$
 In order to
simplify our notation, we denote by
\begin{eqnarray*}
&& {\hat X} (t) \triangleq X (t) - {\bar X} (t) \ , \\
&& {\hat b} (t) \triangleq  b ( t, {\bar X}' (t), {\bar X} ( t  ) )
-  {\bar b} ( t, {\bar X}'(t), {\bar X} ( t) ) \ , \\
&&{\hat g} (t) \triangleq  g ( t, {\bar X}' (t), {\bar X} ( t ) )- {\bar g} ( t, {\bar X}' (t), {\bar X} ( t ) ) \ .
\end{eqnarray*}
Using It\^o's formula to $\| {\hat X} (t) \|_H^2,$
we get that
\begin{eqnarray}
  \begin{split}
    ||\hat X(t)||^2_H=&||\hat x||_H -2 \int_0^t \left < A (s) \hat X (s),
    \hat X(s) \right > d s
+2\int_0^t\Big(\mathbb E' [b ( s, X' (s), X ( s) )
- \bar b ( s, \bar X' (s), \bar X ( s) )], \hat X(s) \Big)_H d s
 \\&+2\int_0^t \Big( \mathbb E'[g ( s, X' (s), X (s))- \bar g(s, \bar X'(s), \bar X(s))], \hat X(s)) \Big)_H d W (s)
 \\&+\int _0^t ||\mathbb E' [ g ( s, X' (s), X ( s  )-\bar g(s, \bar X'(s), \bar X(s)) ]||^2_H ds
  \end{split}
\end{eqnarray}
 In view of  Assumption \ref{ass:1.1} and the elementary inequality $2
a b \leq a^2 + b^2$,  $\forall a, b > 0$, we
obtain
\begin{eqnarray}\label{eq:2.10}
\begin{split}
&\| {\hat X} (t) \|_H^2 + 2 \alpha \int_0^t \| {\hat X} (s) \|_V^2 d s \\
& \leq \| {\hat x}\|_H^2 + K ( C, \lambda ) \int_0^t \| {\hat X} (s) \|_H^2 d s
+ K (C) \int_0^t \mathbb E\|{\hat X} ( s ) \|_H^2 d s + \int_0^t
 \mathbb E'\| {\hat b} (s) \|_H^2 d s
  \\
& \quad + 2 \int_0^T
\mathbb E'\| {\hat g} (s) \|_H^2 d s + 2 \int_0^t \Big( \mathbb E'[g ( s, X' (s), X ( s  ))- \bar g(s, \bar X'(s), \bar X(s))], \hat X(s)) \Big)_Hd W (s) \ .
\end{split}
\end{eqnarray}
Taking expectations on both sides of \eqref{eq:2.10} leads to
\begin{eqnarray}\label{eq:113}
\begin{split}
& {\mathbb E} [ \| {\hat X} (t) \|_H^2 ]
+ 2 \alpha {\mathbb E} \bigg [ \int_0^t \| {\hat X} (s) \|_V^2 d s \bigg ] \\
&\leq {\mathbb E} \| {\hat X}\|_H^2
+ K ( \lambda, C ) {\mathbb E} \bigg [ \int_0^t \| {\hat X} (s) \|_H^2 d s \bigg ]+ \bar {\mathbb E} \bigg [ \int_0^T \| {\hat b} (s) \|_H^2 d s \bigg ]
+ 2\bar  {\mathbb E} \bigg [ \int_0^T \| {\hat g} (s) \|_H^2 d s \bigg].
\end{split}
\end{eqnarray}
Then applying Gr\"onwall's inequality to \eqref{eq:113} yields
\begin{eqnarray}\label{eq:2.12}
\begin{split}
  &\sup_{0\leq t\leq T}  {\mathbb E} [ \| {\hat X} (t) \|_H^2 ] +
{\mathbb E} \bigg [ \int_0^T \| {\hat X} (t) \|_V^2 d t \bigg ]
\\ & \quad \leq K \bigg \{ {\mathbb E} [ || \hat X
\|_H^2 ] + \bar {\mathbb E} \bigg [ \int_0^T \| {\hat b} (t) \|_H^2 d t \bigg ]
+ \bar {\mathbb E} \bigg [ \int_0^T \| {\hat g} (t) \|_H^2 d t \bigg ] \bigg \} \ .
\end{split}
\end{eqnarray}
where $K$ is a positive constant depending only on $T$, $C$, $\alpha$ and
$\lambda$.

Furthermore,  in view of  \eqref{eq:2.10} and \eqref{eq:2.12}, the Lipschitz continuity condition ( see\eqref{eq:2.40}) and the Burkholder-Davis-Gundy, we get  that
\begin{eqnarray}\label{eq:2.145}
\begin{split}
& {\mathbb E} \bigg [ \sup_{0 \leq t \leq T} \| {\hat x} (t) \|_H^2 \bigg ]\\
& \leq K \bigg \{ {\mathbb E} [ \| {\hat x} \|_H^2 ]
+ \bar {\mathbb E} \bigg [ \int_0^T \| {\hat b} (t) ||_H^2 d t \bigg ]
+ \bar {\mathbb E} \bigg [ \int_0^T \| {\hat g} (t) \|_H^2 d t \bigg ] \bigg\}  \\
& \quad + 2 {\mathbb E} \bigg [ \sup_{0 \leq t \leq T} \bigg | \int_0^t \Big( \mathbb E'[g ( s, X' (s), X ( s  ))- \bar g(s, \bar X'(s), \bar X(s))], \hat X(s)) \Big)_H d W (s) \bigg | \bigg ] \\
& \leq K \bigg \{ {\mathbb E} [ \| x \|_H^2 ]
+\bar  {\mathbb E} \bigg [ \int_0^T \| {\hat b} (t) \|_H^2 d t \bigg ]
+ \bar {\mathbb E} \bigg [ \int_0^T \| {\hat g} (t) \|_H^2 d t \bigg ] \bigg\} + \frac{1}{2} {\mathbb E} \bigg [ \sup_{0 \leq t \leq T} \| {\hat X} (t) \|_H^2 \bigg ] \ .
\end{split}
\end{eqnarray}
Therefore, \eqref{eq:3.111} can be obtained by  combining \eqref{eq:2.12} and \eqref{eq:2.145}.
The proof is complete.
\end{proof}

\begin{thm}[{\bf Existence and uniqueness theorem of MF-SPDE}]\label{lem:1.1}

Let Assumption \ref{ass:1.1} be satisfied.
Then for any given initial value $x,$
the MF-SPDE  \eqref{eq:2.1} admits a unique
solution $X (\cdot) \in S_{\mathscr F}^2 ( 0, T; H )$.
\end{thm}

\begin{proof}
The uniqueness of the solution
of MF-SPDE \eqref{eq:2.1} is
implied by the a priori estimate
\eqref{eq:3.111}.
 Consider a family of MF-SPDE
 parameterized by  $\rho \in [0, 1]$ as
 follows:
\begin{eqnarray}\label{eq:3.12}
\begin{split}
 X(t) =& \ x - \int_0^t A (s) X (s) d s + \int_0^t \Big [\rho \mathbb E'[ b ( s, X'(s), X(s) )]+ b_0(t)]d s+\int_0^t \Big[\rho\mathbb E'[g ( s, X' (s), X ( s ) )]+g_0(t)\Big ] d W (s) ,
\end{split}
\end{eqnarray}
where $b_0 (\cdot) \in {\cal M}_{\mathscr F}^2 ( 0, T; H )$ and
$g_0 (\cdot) \in {\cal M}_{\mathscr F}^2 ( 0, T; H )$ are two  any  given stochastic process.
It is easily seen that the original  MF-SPDE
\eqref{eq:2.1} is ''embedded'' in
the MF-SPDE \eqref{eq:3.12}
when we take the parameter $\rho=1$ and
$b_0 (\cdot)\equiv 0,g_0 (\cdot)\equiv0.$
Obviously, the MF-SPDE \eqref{eq:3.12} have  coefficients $(A, \rho b + b_0, \rho g+ g_0)$
satisfying  Assumption \ref{ass:1.1} with the same Lipshcitz constant $C$. Suppose
 for any
$b_0 (\cdot),g_0 (\cdot) \in {\cal M}_{\mathscr F}^2 ( 0, T; H )$ and  some  parameter $\rho = \rho_0$,
the MF-SPDE \eqref{eq:3.12} admits a unique solution $X (\cdot)\in {\cal M}_{\mathscr F}^2 ( 0, T; V)$ . For any parameter $\rho$,
we can rewrite the MF-SPDE \eqref{eq:3.12}  as
\begin{eqnarray}\label{eq:3.14}
\begin{split}
 X(t) =& \ x - \int_0^t A (s) X (s) d s + \int_0^t \Big [\rho_0 \mathbb E'[ b ( s, X'(s), X(s) )]+ b_0(t)+(\rho-\rho_0) \mathbb E'[ b ( s, X'(s), X(s) )]\Big]d s
 \\&+\int_0^t \Big[\rho_0\mathbb E'[g ( s, X' (s), X ( s ) )]+g_0(t)
 +(\rho-\rho_0)\mathbb E'[g ( s, X' (s), X ( s ) )]\Big ] d W (s) .
\end{split}
\end{eqnarray}
 Therefore, by our above supposition,
  for any $x (\cdot) \in {\cal M}_{\mathscr F}^2 ( 0, T; V),$
the following MF-SPDE
\begin{eqnarray}\label{eq:3.19}
\begin{split}
 X(t) =& \ x - \int_0^t A (s) X (s) d s + \int_0^t \Big [\rho_0 \mathbb E'[ b ( s, X'(s), X(s) )]+ b_0(t)+(\rho-\rho_0) \mathbb E'[ b ( s, x'(s), x(s) )]\Big]d s
 \\&+\int_0^t \Big[\rho_0\mathbb E'[g ( s, X' (s), X ( s ) )]+g_0(t)
 +(\rho-\rho_0)\mathbb E'[g ( s, x' (s), x ( s ) )]\Big ] d W (s) .
\end{split}
\end{eqnarray}
 admits a unique solution $X (\cdot) \in {\cal M}_{\mathscr F}^2 ( 0, T; V)$.
 Consequently, now we can define a mapping from $ {\cal M}_{\mathscr F}^2 ( 0, T; V)$ onto itself
and denote by $X (\cdot) = \Gamma (x (\cdot))$.

 In view of the Lipschitz
continuity of $b$ and $g$ and  a priori estimate \eqref{eq:3.111}, for any $x_i (\cdot) \in {\cal M}_{\mathscr F}^2 ( 0, T; V)$, $i = 1, 2$,
we obtain
\begin{eqnarray*}
|| \Gamma ( x_1 (\cdot) ) - \Gamma ( x_2 (\cdot) ) ||^2_{{\cal M}_{\mathscr F}^2 ( 0, T; V)}
&=& || X_1 (\cdot) - X_2(\cdot) ||^2_{{\cal S}_{{\cal M}_{\mathscr F}^2 ( 0, T; V)} }\\
&\leq& K |\rho-\rho_0| \cdot || x_1 (\cdot) - x_2(\cdot) ||^2_{{\cal S}_{{\cal M}_{\mathscr F}^2 ( 0, T; V)} } \ ,
\end{eqnarray*}
Here $K \triangleq K (T, C, \lambda, \alpha)$ is a  positive constant independent of $\rho$. Set $\theta=\frac{1}{2K}.$ Then we  conclude that as long as $| \rho - \rho_0 |^2\leq \theta $,  the mapping $\Gamma$ is
a contraction in ${\cal M}_{\mathscr F}^2 ( 0, T; V)$
which implies  that MF-SPDE \eqref{eq:3.12}
is solvable. It is well-known that
 the MF-SPDE \eqref{eq:3.12} with $\rho_0 = 0$ admits a unique solution by the classic
 existence and uniqueness theory  of SPDE (see e.g. Prévôt and Röckner (2007)). Starting from $\rho = 0$, one can reach $\rho = 1$ in finite steps and this finishes
 the proof of solvability of the MF-SPDE \eqref{eq:3.12}.
Moreover,  from Lemma \ref{lem:1.2} and  the a priori estimate \eqref{eq:3.10},
we obtain $ X(\cdot) \in {\cal S}_{\mathscr F}^2 ( 0, T; V )$. This completes the proof.

\end{proof}
 We conclude this section by
studying another type of MF-SPDE
  in the following abstract stochastic evolution form:
 \begin{eqnarray}\label{eq:3.19}
\left
\{\begin{aligned}
d X(t) = & \ \big\{ - A (t) X (t)+ b ( t,  EX (t), X ( t  ) \big\} d t
+ g( t,
\mathbb EX(t), X (t) ) d W(t) \ , \quad t \in [ 0, T ] \ , \\
X (0) = & \ x \in  H \ ,
\end{aligned}
\right.
\end{eqnarray}
where the coefficients
 \begin{eqnarray} \label{eq:3.20}
  A&=&A(t,\omega) : [ 0, T ] \times \nonumber
\Omega \rightarrow {\mathscr L} ( V, V^* ),\nonumber
\\b&=&b(t,\omega, x', x) : [0, T] \times
\Omega \times H \times H \rightarrow H,
\nonumber
\\g&=&g(t,\omega, x', x) : [0, T] \times
\Omega \times H \times H \rightarrow H
\end{eqnarray}
are given random mappings.

We make the  following  standard assumptions on the coefficients $(A,f, \xi).$
\begin{ass}\label{ass:3.2}
Suppose that there exist constant
$\alpha>0, \lambda, $ and $C$
such that the following  conditions
holds for all $x,x', \bar x, \bar x'\in H$
and  a.e. $( t,\omega ) \in [ 0, T ] \times \Omega$
\begin{enumerate}
\item[(i)](Measurability) The operator $A$ is ${\mathscr P} / {\mathscr B} ( {\mathscr L} ( V, V^* ) )$
measurable; $b$ and $g $ are ${\mathscr P} \otimes
{\mathscr B} (H) \otimes {\mathscr B} (H) / {\mathscr B} (H)$ measurable;
\item[(ii)] (Integrality) $b ( \cdot, 0, 0 ), g ( \cdot, 0, 0 )\in {\cal M}_{\bar {\mathscr F}}^2 ( 0, T; H )$;
\item[(iii)] (Coercivity)
\begin{eqnarray} \label{eq:2.22}
\left < A (t) x, x \right > + \lambda \| x \|^2_H \geq \alpha \| x \|^2_V \  ;
\end{eqnarray}
\item[(iv)](Boundedness)
\begin{eqnarray}
\sup_{( t, \omega ) \in [0, T] \times \Omega} \| A ( t,\omega ) \|_{{\mathscr L} ( V, V^* )} \leq C \ ;
\end{eqnarray}
\item[(iv)] (Lipschitz Continuity)
\begin{eqnarray}\label{eq:2.40}
\| b ( t, x', x) - b ( t, {\bar x'}, {\bar x} ) \|_H
+ \| g ( t, x', x ) - g ( t, {\bar x'}, {\bar x} ) \|_H
\leq C [ \| x - {\bar x} \|_H + \| x' - {\bar x}' \|_H ] \ .
\end{eqnarray}
\end{enumerate}
\end{ass}
Similar to  Theorem \ref{lem:1.2} and
 \ref{lem:1.1},  we have
the following two
important results on the
solution to MF-SPDE \eqref{eq:3.19}.

\begin{thm}\label{lem:3.3}
Let Assumption \ref{ass:3.2} be satisfied.
Then for any given initial value $x,$
the MF-SPDE  \eqref{eq:3.19} has a unique
solution $X (\cdot) \in S_{\mathscr F}^2 ( 0, T; H )$.
\end{thm}

\begin{thm} \label{thm:3.4}
Let Assumption \ref{ass:3.2} be satisfied.
    Suppose that $X (\cdot)$ be the solution to MF-SPDE \eqref{eq:3.19} with initial value $x\in H.$
   Then  the following estimate holds:
\begin{eqnarray}\label{eq:c6}
\begin{split}
 {\mathbb E} \bigg [ \sup_{0 \leq t \leq T} \|X(t) \|_H^2 \bigg
]
+  {\mathbb E} \bigg [ \int_0^T \| X (t) \|_V^2 d t \bigg ]
 \leq K \bigg \{  {\mathbb E} [ \| x\|^2_H ] + \bar {\mathbb E}
\bigg [ \int_0^T \| b ( t, 0, 0 ) \|_H^2 d t \bigg ] + {\mathbb E}
\bigg [ \int_0^T \| g ( t, 0, 0 ) \|_H^2 d t \bigg ]\bigg \} \ ,
\end{split}
\end{eqnarray}
where $K$  is a positive constant depending only on $T, C, \alpha$ and
$\lambda.$
 Further, suppose that  ${\bar X} (\cdot)$ is  the
solution to MF-SPDE  \eqref{eq:3.19}
 with  the coefficients  $( A, {\bar b}, {\bar g})$ satisfying Assumption \ref{ass:3.2}
 and the  initial value $\bar x\in H.$
   Then we have
\begin{eqnarray}\label{eq:c2}
\begin{split}
& {\mathbb E} \bigg [ \sup_{0 \leq t \leq T} \| X(t) - {\bar X} (t) \|_H^2 \bigg ]
+ {\mathbb E} \bigg [ \int_0^T \| X (t) - {\bar X} (t) \|_V^2 d t \bigg ]  \\
& \leq K \bigg \{ {\mathbb  E}\bigg [ \int_0^T \| b ( t, \mathbb E{\bar X} (t), {\bar X} ( t) )
- {\bar b}( t, \mathbb E{\bar X} (t), {\bar X} ( t ) ) \|_H^2 d t \bigg ] + {{\mathbb E}} \bigg [ \int_0^T \| g ( t, \mathbb E{\bar X} (t),
{\bar X} ( t ) ) - {\bar g}( t, \mathbb E{\bar X} (t), {\bar X} ( t ) ) \|_H^2 d t \bigg ]\bigg\} \ .
\end{split}
\end{eqnarray}
\end{thm}

\section{ Mean-Field Backward Stochastic
Partial  Differential Equation}

In this section, in Gelfand triple $(V, H, V^*),$ we begin to investigate the
MF-BSPDE in the following abstract stochastic evolution form:
\begin{eqnarray}\label{eq:2.7}
\left\{
\begin{aligned}
d Y(t) =& \ \big[ A (t) Y (t) +  \mathbb E'[ f ( t, Y '(t), Z' (t), Y  (t), Z  (t) ) \big]d t  + Z (t) d W (t) ,  \quad t \in [0, T] , \\
Y(T) =& \ \xi ,
\end{aligned}
\right.
\end{eqnarray}
where the coefficients $( A, f, \xi )$ are  the  following mappings
\begin{eqnarray} \label{eq:4.2}
  A&=&A(t,\omega) : [ 0, T ] \times \nonumber
\Omega \rightarrow {\mathscr L} ( V, V^* ),\nonumber
\\f&=&f(t,\omega', \omega, y',z', y, z) : [0, T] \times
\bar \Omega \times V \times H \times V \times H \rightarrow H,
\nonumber
\\
\xi&=&\xi(\omega): \Omega \rightarrow H.
\end{eqnarray}
In the above, we have used the following notation defined by
\begin{eqnarray}
   \mathbb E'[f ( t,  Y' (t), Z'( t),  Y (t), Z( t) ]=
   \int_\Omega f(t, \omega', \omega, Y'(t,\omega'),Z'(t,\omega'),
    Y(t,\omega), Z(t,\omega')) P(d \omega')
\end{eqnarray}
 Furthermore, we make
 the following
 standard assumption
 on the coefficients $(A, f, \xi).$

\begin{ass}\label{ass:1.3}
Suppose that there exist constant
$\alpha>0, \lambda, $ and $C$
such that the following  conditions
holds
for all $( y', z', y, z ), ( {\bar y'}, {\bar z'}, {\bar y},
{\bar z} ) \in V \times H \times V \times H$
and a.e. $( t, \omega', \omega ) \in [ 0, T ] \times \bar\Omega.$
\begin{enumerate}
\item[(i)](Measurability) The operator $A$ is
${\mathscr P} / {\mathscr B} ( {\mathscr L} ( V, V^* ) )$-measurable;
$f$ is $\bar {\mathscr P} \otimes {\mathscr B} (V) \otimes {\mathscr B}(H) \otimes {\mathscr B} (V) \otimes
{\mathscr B} (H) / {\mathscr B} (H)$-measurable; $\xi$ is ${\mathscr F}_T$-measurable.;
\item[(ii)] (Integrality)  $f ( \cdot, 0, 0, 0, 0 ) \in {\cal M}_{\mathscr F}^2 ( 0, T; H )$ and
$\xi \in {\cal L}^2 ( {\mathscr F}_T; H )$;
\item[(ii)] (Coercivity)
\begin{eqnarray} \label{eq:2.22}
\left < A (t) x, x \right > + \lambda \| x \|^2_H \geq \alpha \| x \|^2_V \  ;
\end{eqnarray}
\item[(iv)](Boundedness)
\begin{eqnarray}
\sup_{( t, \omega ) \in [0, T] \times \Omega} \| A ( t,\omega ) \|_{{\mathscr L} ( V, V^* )} \leq C \ ;
\end{eqnarray}
\item[(iv)] (Lipschitz Continuity)

\begin{equation}\label{eq:2.18}
\| f ( t, y', z',  y, z ) - f ( t, {\bar y'}, {\bar z'}, {\bar y}, {\bar z}) \|_H
\leq C ( \| y' - {\bar y'} \|_V + \| z' - {\bar z'} \|_H + \| y - {\bar y} \|_V
+ \| z - {\bar z}\|_H ) .
\end{equation}

\end{enumerate}
\end{ass}
Now we give  the definition of  the solutions to
MF-SPDE \eqref{eq:2.7}.
\begin{defn}\label{defn:c1}
A  $(V \times H)$-valued, ${\mathbb F}$-adapted process pair $( Y (\cdot), Z (\cdot) )$
is said to be  a solution to the MF-BSPDE \eqref{eq:2.7}, if $Y (\cdot) \in {\cal M}_{\mathscr F}^2 ( 0, T; V )$
and $ Z (\cdot) \in {\cal M}_{\mathscr F}^2 ( 0, T; H )$ such that
\begin{eqnarray}\label{eq:c5}
( Y (t),  \phi )_H  &=& (\xi, \phi)_H
- \int_t^T  \mathbb E'[ f ( s, Y' (s), Z' (s), Y(s),
Z (s) ),  \phi )_H ]d s \nonumber \\
&& - \int_t^T \left < A (s) Y(s), \phi \right > d s - \int_t^T ( Z (s), \phi )_H d W (s) , \quad t \in [0, T] ,
\end{eqnarray} holds for every $\phi \in V$ and
a.e. $( t, \omega ) \in [ 0, T ] \times \Omega$,
or alternatively,
in the sense of $V^*$, $( Y (\cdot), Z (\cdot) )$
satifies  the following It\^o form:
\begin{eqnarray}
Y(t) &=& \xi - \int_t^T {\mathbb E'} [ f ( s, Y' (s), Z' (s), Y(s), Z(s) )] d s \nonumber \\
&& - \int_t^T A(s) Y (s) d s - \int_t^T Z (s) d W (s) ,  \quad t \in [0, T] .
\end{eqnarray}
\end{defn}
The following result is
the continuous dependence theorem for the solution
 to the MF-BSPDE \eqref{eq:2.7}
  with respect to the coefficients $(A,f, \xi),$  which also is referred to as a priori
estimate for the solution.
\begin{thm}[{\bf Continuous dependence theorem of MF-BSPDE}]\label{lem:1.4}
Suppose that  $( Y (\cdot), Z (\cdot) )$
is the solution to the MF-BSPDE \eqref{eq:2.7}
 with the coefficients $( A, f, \xi )$ satisfying Assumption \ref{ass:1.3}, then
 we have the following a priori estimate
\begin{eqnarray}\label{eq:2.15}
&& {\mathbb E} \bigg [ \sup_{0 \leq t \leq T} \| Y (t) \|^2_H \bigg ]
+ {\mathbb E} \bigg [ \int_0^T \| Y (t) \|_V^2 d t \bigg ]
+ {\mathbb E} \bigg [ \int_0^T \| Z (t) \|^2_H d t \bigg ] \nonumber \\
&& \leq K \bigg \{ {\mathbb E} [ \| \xi \|^2_H ] + \bar {\mathbb E} \bigg
[ \int_0^T \| f ( t, 0, 0, 0, 0 ) \|^2_H d t \bigg ] \bigg \} .
\end{eqnarray}
Here  $K \triangleq K ( T, C, \alpha, \lambda )$ is a positive constant depending only on $T$, $C$, $\alpha$ and $\lambda$.
 Assume that $( {\bar Y} (\cdot), {\bar Z} (\cdot) )$ is the solution to the
MF-BSPDE\eqref{eq:2.7}  with
the coefficients $(A^*, {\bar f}, {\bar \xi})$
 satisfying Assumption \ref{ass:1.3},
then it holds that
\begin{eqnarray}\label{eq:2.14}
&& {\mathbb E} \bigg [ \sup_{0 \leq t \leq T} \| Y (t) - {\bar Y} (t) \|_H^2 \bigg ]
+ {\mathbb E} \bigg [ \int_0^T \| Y (t) - {\bar Y} (t) \|_V^2 d t \bigg ]
+ {\mathbb E} \bigg [ \int_0^T \| Z (t) - {\bar Z} (t) \|^2_H d t \bigg ] \nonumber \\
&& \leq K \bigg \{ \mathbb  E [ \| \xi - {\bar \xi} \|_H^2 ]
+ \bar {\mathbb E} \bigg [ \int_0^T \| f ( t, {\bar Y'} (t), {\bar Z'} (t),{\bar Y}  (t),
{\bar Z}  (t) )- {\bar f} ( t, {\bar Y'} (t), {\bar Z'} (t),{\bar Y}  (t),{\bar Z} (t) ) \|^2_H d t \bigg ] \bigg \} .
\end{eqnarray}
\end{thm}

\begin{proof}
If we take the coefficients $( A, {\bar f}, {\bar \xi} ) = (A, 0, 0)$,
then the corresponding solution to the
MF-BSPDE \eqref{eq:2.7} is $(\bar Y(\cdot),\bar Z(\cdot)) = (0, 0)$
and the estimate \eqref{eq:2.15} follows from the estimate \eqref{eq:2.14} immediately.
Therefore, it suffices to prove that \eqref{eq:2.14} holds. To simplify our notation, we define
\begin{eqnarray*}
&& {\hat Y} (t) \triangleq Y (t) - {\bar Y} (t) , \quad
{\hat Z} (t) \triangleq Z (t) - {\bar Z} (t) , \quad
{\hat \xi} \triangleq \xi - {\bar \xi} , \\
&& {\hat f} (t) \triangleq f ( t, {\bar Y'} (t), {\bar Z'} (t), {\bar Y} (t),
{\bar Z} (t) )- {\bar f} ( t, {\bar Y'} (t), {\bar Z'} (t),{\bar Y} (t),{\bar Z} (t) ) .
\end{eqnarray*}
Using It\^o's formula to $\| {\hat Y} (t) \|_H^2$
and Assumption \ref{ass:1.3} and the classic inequality  leads to  $2
a b \leq \frac{1}{\epsilon} a^2 + \epsilon b^2$, $\forall a, b > 0$,
$\epsilon > 0$, we  have
\begin{eqnarray}\label{eq:2.24}
&& \| {\hat Y} (t) \|_H^2 + 2 \alpha \int_t^T \| {\hat Y} (s) \|_V^2 d s + \int_t^T \| {\hat Z} (s) \|_H^2 d s \nonumber \\
&& \leq \| {\hat \xi} \|_H^2 + K (  C,\lambda, \epsilon ) \int_t^T \| {\hat Y} (s) \|_H^2 d s
+ \epsilon \int_t^T \| {\hat Y} (s) \|_V^2 d s + \epsilon \int_t^T \| {\hat Z} (s) \|_H^2 d s \nonumber \\
&& \quad  +K ( C, \epsilon ) \mathbb E
\bigg[\int_t^T \|{\hat Y} (s) \|_H^2 d s\bigg] + \epsilon \mathbb E\bigg[\int_t^T \|
{\hat Y}(s) \|_V^2 d s\bigg]
+ \epsilon \mathbb E\bigg[\int_t^T \|{\hat Z} (s) \|_H^2 d s\bigg] \nonumber \\
&& \quad +
 \mathbb E'\int_t^T \| {\hat f} (s) \|_H^2 d s - 2 \int_t^T ( {\hat Y} (s), {\hat Z} (s) )_H d W (s).
\end{eqnarray}

By taking expectations on both sides
of \eqref{eq:2.24}  and  taking $\epsilon$ small enough such that $2 \alpha - 2 \epsilon
> 0$ and $1 - 2 \epsilon > 0,$
we have
\begin{eqnarray}\label{leq:c4}
&& {\mathbb E} [ \| {\hat Y} (t) \|_H^2 ]
+ {\mathbb E} \bigg [ \int_t^T \| {\hat Y} (s) \|_V^2 d s \bigg ]
+ {\mathbb E} \bigg [ \int_t^T \| {\hat Z} (s) \|_H^2 d s \bigg ] \nonumber \\
&& \leq K ( T, C, \alpha, \lambda )  \bigg \{ \mathbb E \| {\hat \xi} \|_H^2 +
\bar {\mathbb E}\int_t^T \| {\hat f} (s) \|_H^2 d s + \mathbb E\int_t^T \| {\hat Y}(s) \|_H^2 d s \bigg \} .
\end{eqnarray}
 Here $K ( T, C, \alpha, \lambda )$ is a general positive constant depending on $\alpha$, $T$, $C$, and $\lambda$.

Then by Gr\"onwall's inequality to \eqref{leq:c4},
we obtain
\begin{eqnarray}\label{leq:c5}
&& \sup_{0 \leq t \leq T} {\mathbb E} [ \| {\hat Y} (t) \|_H^2 ]
+ {\mathbb E} \bigg [ \int_0^T \| {\hat Y} (t) \|_V^2 d t \bigg ]
+ {\mathbb E} \bigg [ \int_0^T \| {\hat Z} (t) \|_{H}^{2} d t \bigg ] \nonumber \\
&& \leq K ( T, C, \alpha, \lambda ) \bigg \{ {\mathbb E} [ \| {\hat \xi} \|_H^2 ]
+\bar  {\mathbb E} \bigg [ \int_0^T \| {\hat f} (t) \|_H^2 d t \bigg ] \bigg \} .
\end{eqnarray}
 In view of \eqref{eq:2.24}, \eqref{leq:c5} and the Burkholder-Davis-Gundy inequality, we have
\begin{eqnarray}\label{eq:2.29}
{\mathbb E} \bigg [ \sup_{0 \leq t \leq T} \| {\hat Y} (t) \|_H^2 \bigg ]
&\leq& K ( T, C, \alpha, \lambda ) \bigg \{ {\mathbb E} [ \| {\hat \xi} \|_H^2 ]
+ \bar {\mathbb E} \bigg [ \int_0^T \| {\hat f} (t) \|_H^2 d t \bigg ] \bigg \} \nonumber \\
&&+ 2 {\mathbb E} \bigg [ \sup_{0 \leq t \leq T} \bigg | \int_t^T (
{\hat Y} (s), {\hat Z} (s) )_H d W (s) \bigg | \bigg ] \nonumber \\
&\leq & K ( T, C, \alpha, \lambda ) \bigg \{ {\mathbb E} [ \| {\hat \xi} \|_H^2 ] + {\mathbb E}
\bigg [ \int_0^T \| {\hat f} (t) \|_H^2 d t \bigg ] \bigg \}
\nonumber
\\&&+ \frac{1}{2} {\mathbb E} \bigg [ \sup_{0 \leq t \leq
T} \| {\hat Y} (t) \|_H^2 \bigg ] .
\end{eqnarray}
which implies that
\begin{eqnarray}\label{eq:2.29}
{\mathbb E} \bigg [ \sup_{0 \leq t \leq T} \| {\hat Y} (t) \|_H^2 \bigg ]
&\leq& K ( T, C, \alpha, \lambda ) \bigg \{ {\mathbb E} [ \| {\hat \xi} \|_H^2 ]
+ \bar {\mathbb E} \bigg [ \int_0^T \| {\hat f} (t) \|_H^2 d t \bigg ] \bigg \}
\end{eqnarray}
There we conclude that \eqref{eq:2.14} holds
by \eqref{eq:2.29} with \eqref{leq:c5}
The proof is complete.
\end{proof}

\begin{thm}[{\bf Existence and uniqueness theorem of MF-BSPDE}]\label{lem:1.3}
Let the coefficients $(A, f, \xi)$ satisfy Assumption \ref{ass:1.3},
 then MF-BSPDE \eqref{eq:2.7} admits a unique solution $( Y (\cdot), Z (\cdot) )
\in {\cal S}_{\mathscr F}^2 ( 0, T; V ) \times {\cal M}_{\mathscr F}^2 ( 0, T; H )$.
\end{thm}

\begin{proof}

The uniqueness of the solution
of MF-SPDE \eqref{eq:2.7} is
implied by the a priori estimate
\eqref{eq:2.14}.
 Consider a family of MF-BSPDE
 parameterized by  $\rho \in [0, 1]$ as
 follows:
\begin{eqnarray}\label{eq:1.7}
Y(t) &=& \xi - \int_t^T \bigg \{ A (s) Y (s) + \rho \mathbb E'[ f ( s, Y' (s), Z' (s), Y(s), Z  (s) ) ] \nonumber \\
&& + f_0 (s) \bigg \} d s -\int_t^T Z (s) d W (s) .
\end{eqnarray}
 where $f_0 (\cdot) \in {\cal M}^2_{{\mathscr F}} ( 0, T; H )$ is  an  arbitrary stochastic process.

 It is easily seen that the original  MF-SPDE
\eqref{eq:2.7} is ''embedded'' in
the MF-SPDE \eqref{eq:1.7}
when we take the parameter $\rho=1$ and
$f_0 (\cdot)\equiv 0.$
Obviously, the MF-SPDE \eqref{eq:1.7} have  coefficients $(A, \rho f + b_0, \xi)$
satisfying  Assumption \ref{ass:1.1}.
Suppose for some $\rho=\rho_0$ and  any $f_0 \in {\cal M}^2_{\mathscr F} (0, T; H)$, the MF-BSPDE \eqref{eq:1.7} admits a unique solution $( Y (\cdot), Z (\cdot) )
\in {\cal M}_{\mathscr F}^2 ( 0, T; V ) \times {\cal M}_{\mathscr F}^2 (0, T; H)$.
Then for any  $\rho$,   we can rewrite the MF-SPDE \eqref{eq:1.7} as follows
\begin{eqnarray}\label{eq:1.8}
Y(t) &=& \xi - \int_t^T \bigg\{ A (s) Y (s) + \rho_0 \mathbb E'[f ( s, Y' (s), Z' (s), Y (s),Z (s) )]
\nonumber \\
&& + f_0 (s) + ( \rho - \rho_0 )
  \mathbb E'[f ( s, Y' (s), Z' (s),  Y (s), Z (s) )] \bigg\} d s \nonumber \\
&& -\int_t^T Z(s) d W (s) .
\end{eqnarray}
Thus by  our above assumption, for any
stochastic process pair $( y (\cdot), z (\cdot) ) \in {\cal M}_{\mathscr F}^2 ( 0, T; V) \times {\cal M}_{\mathscr F}^2 ( 0, T; H )$,
the following MF-BSPDE
\begin{eqnarray}\label{eq:1.9}
Y (t) &=& \xi - \int_t^T \bigg \{ A (s) Y (s)
+ \rho_0 \mathbb E'[f ( s, Y' (s), Z' (s), Y (s),Z (s) )] \nonumber \\
&& + f_0 (s) + ( \rho - \rho_0 )
 \mathbb E'[f ( s, y' (s), z' (s), y (s), z (s) )] d s \nonumber \\
&& - \int_t^T Z (s) d W (s) ,
\end{eqnarray}
 admits a unique solution $( Y (\cdot), Z (\cdot) ) \in {\cal M}_{\mathscr F}^2 ( 0, T; V) \times {\cal M}_{\mathscr F}^2( 0, T; H)$.
 which imply that we can define a mapping  from ${\cal M}_{\mathscr F}^2 ( 0, T; V) \times {\cal M}_{\mathscr F}^2( 0, T; H)$ into
itself denoted by ${\cal I} ( y (\cdot), z (\cdot) ) = ( Y (\cdot), Z (\cdot) )$.

In view of the a priori estimate \eqref{eq:2.14} and the Lipschitz
continuity of $f,$ for any $( y_i (\cdot), z_i (\cdot) ) \in {\cal M}_{\mathscr F}^2 ( 0,
T; V ) \times {\cal M}_{\mathscr F}^2 ( 0, T; H ),$
 it holds that
\begin{eqnarray}
&&\| {\cal I}( x_1 (\cdot), z_1 (\cdot) ) - {\cal I}( _2 (\cdot), z_2 (\cdot) )
\|^2_{ {\cal M}_{\mathscr F}^2 ( 0, T; V) \times {\cal M}_{\mathscr F}^2 ( 0, T; H ) }
\\&&= \| ( Y_1 (\cdot), Z_1 (\cdot) ) - ( Y_2 (\cdot), Z_2 (\cdot) )
\|^2_{ {\cal M}_{\mathscr F}^2 ( 0, T; V) \times {\cal M}_{\mathscr F}^2 ( 0, T; H ) } \nonumber \\
&& \leq K \bar {\mathbb E} \bigg[\int_0^T  \bigg|\rho_0 f ( s, Y_2' (s), Z_2'
(s), Y_2  (s), Z_2 (s) )
+( \rho - \rho_0 )  f ( s, y_1' (s), z_1 '(s),
y_1  (s), z _1(s) ) \nonumber \\
&& \quad -\rho_0 f ( s, Y_2' (s), Z_2' (s),
Y_2  (s), Z_2  (s) )
-( \rho - \rho_0 )  f ( s, y_2' (s), z_2' (s),
y_2 (s), z_2  (s) ) \bigg|^2ds\bigg] \nonumber \\
&& \leq K | \rho - \rho_0 |^2 \times \| ( y_1 (\cdot), z_1 (\cdot) ) - ( y_2 (\cdot), z_2 (\cdot) )
\|^2_{ {\cal M}_{\mathscr F}^2 ( 0, T; V) \times {\cal M}_{\mathscr F}^2 (0, T; H) }.
\end{eqnarray}
 Here we note that $K \triangleq K (T, C, \lambda, \alpha)$ is a constant independent of $\rho$ and
\begin{eqnarray*}
\| ( Y_1 (\cdot), Z_1 (\cdot) ) - ( Y_2 (\cdot), Z_2 (\cdot) )
\|^2_{ {\cal M}_{\mathscr F}^2 ( 0, T; V ) \times {\cal M}_{\mathscr F}^2 ( 0, T; H ) }
\triangleq \| Y_1 (\cdot) - Y_2 (\cdot) \|^2_{ {\cal M}_{\mathscr F}^2 ( 0, T; V ) }
+ \| Z_1 (\cdot) - Z_2 (\cdot) \|^2_{ {\cal M}_{\mathscr F}^2 ( 0, T; H ) } .
\end{eqnarray*}
Set $\theta=\frac{1}{2K}.$ Then we  conclude that as long as $| \rho - \rho_0 |^2 \leq \theta $,  the mapping ${\cal I}$ is
a contraction in ${\cal M}_{\cal F}^2 ( 0, T; V ) \times {\cal M}_{\cal F}^2 ( 0, T; H )$
which implies  that MF-BSPDE \eqref{eq:1.7}
is solvable.
 In view of Proposition 3.2 in Du and Meng (2010), we know that the MF-BSPDE \eqref{eq:1.7} with $\rho_0 = 0$ admits a unique solution. Now we can start from $\rho = 0$ and then  reach $\rho = 1$ in finite steps which  finishes
 the proof of solvability of the MF-BSPDE \eqref{eq:1.7}.
Moreover,  from  the a priori estimate \eqref{eq:2.15},
we obtain $Y (\cdot) \in {\cal S}_{\cal F}^2 ( 0, T; H )$. This completes the proof.

\end{proof}

\section{Optimal Control  of
Mean-Field Stochastic Partial Differential Equation}

\subsection{Formulation of the Optimal Control Problem}

In this subsection, we present our optimal
control problem studied in this paper.
Firstly, in the Gelfand triple $( V, H, V^* ),$ consider the following controlled system:
\begin{eqnarray}\label{eq:5.1}
\left\{
\begin{aligned}
d X (t) =& \ [ - A (t) X(t) + h ( s, X (s),
\mathbb E [X (s)], u(s) ) ] d t + g ( s, X (s),\mathbb E [X (s)],u(s) ) d W (t) , \quad t \in [ 0, T ] , \\
X (0) =& \ x\in H,
\end{aligned}
\right.
\end{eqnarray}
with the cost functional
\begin{eqnarray}\label{eq:5.2}
J ( u (\cdot) ) = {\mathbb E} \bigg [ \int_0^T l ( s, X (s),
\mathbb E [X (s)] ) d t
+ \Phi ( X (T), \mathbb E [X(T)]) \bigg ] .
\end{eqnarray}
In the above,  $A: [ 0, T ] \times \Omega
\rightarrow {\mathscr L} ( V, V^* ),$
$h, g: [ 0, T ] \times \Omega \times H \times H \times {\mathscr U} \rightarrow H,$
$l: [ 0, T ] \times \Omega \times H \times H \times {\mathscr U} \times {\mathscr U} \rightarrow {\mathbb R},$ $\Phi: \Omega \times H
\times H\rightarrow {\mathbb R}.$

Let us make the following assumptions.

\begin{ass}\label{ass:2.5}
\begin{enumerate}
\item[]
\item[(i)]  ${\mathscr U}$   is a nonempty convex closed subset of a real separable Hilbert space $U$.
\item[(ii)] The operator $A$ is ${\mathscr P} / {\mathscr B} ( {\mathscr L} ( V, V^* ) )$-measurable and satisfies the conditions (iii) and (iv) in Assumption \ref{ass:3.2}.
\item[(iii)]The mappings $h$ and  $g$ are
${\mathscr P} \otimes {\mathscr B} (H) \otimes {\mathscr B} (H) \otimes {\mathscr B} ({\mathscr U}) / {\mathscr B} (H)$-measurable such that $h ( \cdot, 0, 0 ), g ( \cdot, 0, 0, 0 ) \in {\cal
M}^2_{\mathscr F} ( 0, T; H )$. Moreover, for almost all $( t, \omega ) \in [ 0, T ] \times \Omega$,  $h$ and $g$ have  continuous and uniformly bounded G\^ateaux derivatives
$h_x, h_{x'}, g_x, g_{x'}, h_u$ and  $g_u$.
\item[(iv)] The mapppings $l$
is ${\mathscr P} \otimes {\mathscr B} (H) \otimes {\mathscr B} (H) \otimes {\mathscr B} ({\mathscr U})/ {\mathscr B} ({\mathbb R})$-measurable and  $\Phi $
is ${\mathscr F}_T\otimes {\mathscr B} (H)\otimes {\mathscr B} (H) / {\mathscr B} ({\mathbb R})$-measurable. For almost all $( t, \omega ) \in [ 0, T ] \times \Omega$, $l$ has continuous G\^ateaux derivatives $l_x,
l_{x'}$ and $l_u$, $\Phi ( \omega, x )$
has continuous G\^ateaux derivative $\Phi_x$.
Moreover, for almost all $( t, \omega ) \in [ 0, T ] \times \Omega$, there is  a constant $C > 0$
such that
\begin{eqnarray*}
| l ( t, x, x', u,  ) |
\leq  C ( 1 + \| x \|^2_H + \| x'\|^2_H + \| u \|_U^2 ) ,
\end{eqnarray*}
\begin{eqnarray*}
&& \| l_x ( t, x, x', u) \|_H + \| l_{x'} ( t, x, x', u) \|_H
+ \| l_u ( t, x, x', u ) \|_U  \\
&& \qquad\qquad\qquad\qquad\qquad \leq C ( 1 + \| x \|_H + \| x' \|_H + \| u \|_U  ) ,
\end{eqnarray*}
and
\begin{eqnarray*}
& | \Phi (x, x') | \leq C ( 1 + \| x \|^2_H
 + \| x' \|^2_H) , \\
& \| \Phi_x (x, x') \|_H \leq C ( 1 + \| x \|_H++ \| x' \|^2_H ), \forall ( x, x', u ) \in H \times H \times {\mathscr U}.
\end{eqnarray*}
\end{enumerate}
\end{ass}

Now we define
\begin{defn}
A  predictable control process $u (\cdot)$ is said to be admissible if
$u (\cdot) \in {\cal M}^2 ( 0, T; U )$ and $u (t) \in {\mathscr U}$, a.e. $t \in [0, T]$, ${\mathbb P}$-a.s..
 Denote by ${\cal A}$  the set of all admissible control processes.
\end{defn}

Given $u(\cdot),$ equation \eqref{eq:5.1}
is a MF-SPDE with random coefficients.
From Theorem \ref{lem:3.3} and Theorem \ref{thm:3.4}, it is easily
seen that  under  Assumption \ref{ass:2.5}, the state equation \eqref{eq:5.1} admits a unique solution $X (\cdot)\equiv X^u(\cdot) \in {\cal S}^2_{\mathscr F} ( 0, T; H )$ and
the cost functional is well-defined In the case that
$X(\cdot)$ is the solution of  \eqref{eq:3.1}
corresponding to $u(\cdot)\in \cal A,$
We call $( u (\cdot); X (\cdot) )$ an admissible pair, and $X (\cdot) $ an admissible  state process.

Our optimal control problem  can be stated as
follows
\begin{pro}\label{pro:2.1}

Minimizes \eqref{eq:5.2} over $\cal A$
Any ${\bar u} (\cdot) \in {\cal A}$
satisfying
\begin{eqnarray}\label{eq:b8}
J ( {\bar u}(\cdot) ) = \inf_{u (\cdot) \in {\cal A}} J ( u (\cdot) ) .
\end{eqnarray}
\end{pro}
is called an optimal control process of
Problem \ref{pro:2.1}.  The corresponding  state process ${\bar X} (\cdot)$ and the admissible   $( {\bar u} (\cdot); {\bar X} (\cdot) )$
is called an optimal state process
and  optimal pair of
Problem \ref{pro:2.1}, respetively.

For any admissible pair $( u (\cdot); X (\cdot) )$,  the adjoint equation of the
state equation \eqref{eq:5.1} and \eqref{eq:5.4}
is defined  by  the following
BSDE whose  unknown is a pair of
${\mathbb F}$-adapted processes
$(p(\cdot), q(\cdot))$
\begin{eqnarray} \label{eq:4.1}
\left\{
\begin{aligned}
d p(t) = & - \bigg\{ - A^* (t) p (t) +
h^*_x(t, X(t), \mathbb E [X(t)], u(t))p(t)
+\mathbb E [h^*_x(t, X(t), \mathbb E [X(t)], u(t))p(t)]
\\&+g^*_x(t, X(t), \mathbb E [X(t)], u(t))q(t)
+\mathbb E[g^*(t, X(t), \mathbb E [X(t)], u(t))q(t)]
 \\&
 +l_x(t, X(t), \mathbb E[ X(t)], u(t))
+\mathbb E[ l_x(t, X(t), \mathbb E [X(t)], u(t))]\bigg\}d t
 + q (t) d W (t) ,  \quad t \in [0, T] , \\
p (T) = & \ \Phi_{x} ( X (T), \mathbb E [X(T)] )
+\ \mathbb E[\Phi_{x'} ( X (T), \mathbb E [X(T)] ) ] ,
\end{aligned}
\right.
\end{eqnarray}
Indeed, the above equation  is a linear MF-BSPDE
where  $A^*$ is the adjoint operator of $A$.
Further,  we can easily see that $A^*$ also satisfies tha boundedness and coercivity conditions
 In view of Theorem \ref{lem:1.3}, the linear MF-BSPDE \eqref{eq:4.1} have  a unique solution $( p (\cdot), q (\cdot) )
\in {\cal S}^2_{\mathscr F} ( 0, T; V ) \times {\cal M}^2_{\mathscr F} ( 0, T; H )$.

Define the Hamiltonian ${\cal H}: [ 0, T ] \times \Omega \times H \times H \times {\mathscr U} \times {\mathscr U} \times V \times H
\rightarrow {\mathbb R}$  by
\begin{eqnarray}\label{eq4.2}
{\cal H} ( t, x, x', u, p, q ) := \left ( h ( t, x, x', u), p \right )_H
+\left( g ( t, x, x', u), q \right)_H + l ( t, x, x', u ) .
\end{eqnarray}
 Under Assumption \ref{ass:2.5}, we can see that
  the Hamiltonian ${\cal H}$
is also continuously G\^ateaux differentiable in $( x, x', u)$. Denote by  ${\cal H}_x$, ${\cal H}_{x'}$
 and ${\cal H}_u$  the corresponding G\^ateaux derivatives.

Therefore, using the notation of
Hamiltonian ${\cal H},$
the adjoint equation \eqref{eq:4.1} can be
written as

\begin{eqnarray} \label{eq:4.1}
\left\{
\begin{aligned}
d p(t) = & - \{ - A^* (t) p (t) + {\cal H}_{x} (t) + {\mathbb E} [ {\cal H}_{x'}(t)]d t + q (t) d W (t) ,  \quad t \in [0, T] , \\
p (T) = &\ \Phi_{x} ( X (T), \mathbb E [X(T)] )
+\ \mathbb E\Phi_{x'} ( X (T), \mathbb E [X(T)] )
\end{aligned}
\right.
\end{eqnarray}
Here we have used the following
shorthand notation:
\begin{eqnarray}\label{eq:4.3}
{\cal H} (t) \triangleq {\cal H} ( t, X (t),
\mathbb EX ( t), u (t), p (t), q (t) ).
\end{eqnarray}

\subsection{ A Variation Formula for the
Cost Functional}

Suppose that
$(u(\cdot); X(\cdot))$
and $(\bar u(\cdot);\bar X(\cdot))$
are any two given  admissible
control pairs. And let $(  \bar p (\cdot), \bar q (\cdot) )$  the  solution to  the
 corresponding adjoint equation \eqref{eq:4.1}
  associated with the admissible
   control pair $(\bar u(\cdot); \bar X(\cdot))$. In order to simplify our notation in the rest of the paper, we shall use the following shorthand notation
\begin{eqnarray}\label{eq:5.8}
&& \rho (t) \triangleq \rho ( t, X (t),
\mathbb E [X ( t )], u (t) ) , \quad  \rho \triangleq h , g , \nonumber \\
&& {\bar \rho} (t) \triangleq \rho ( t, {\bar X} (t), \mathbb E[{\bar X} ( t )], {\bar u} (t)) ,
\quad \rho \triangleq h , g , \nonumber\\
&& {\cal H} (t) \triangleq {\cal H} ( t, X (t), \mathbb E[X ( t)],  u (t), {\bar p} (t), {\bar q} (t) ) ,
\nonumber \\
&& {\bar {\cal H}} (t) \triangleq {\cal H} ( t, {\bar X} (t), \mathbb E[{\bar X}(t)], {\bar u} (t) ,{\bar p} (t), {\bar q} (t) ) .
\end{eqnarray}

To obtain the variation formula for the
cost functional , we need  the following
basic result.

\begin{lem} \label{lem:4.3}
Let  Assumption \ref{ass:2.5} be satisfied.
Then difference  $J ( u (\cdot) ) - J ( {\bar u} (\cdot) )$ of the cost functionals associated with the two admissible pairs
$(u(\cdot); X(\cdot))$
and $(\bar u(\cdot);\bar X(\cdot))$
 has the following
 representation:
\begin{eqnarray}\label{eq:4.9}
J ( u (\cdot) ) - J ( {\bar u} (\cdot) )
&=& {\mathbb E} \bigg [ \int_0^T \bigg \{ {\cal H} (t) - {\bar {\cal H}} (t)
- ( {\bar {\cal H}}_x (t) +
 \mathbb E [{\bar {\cal H}}_{x'} (t)],
 X(t) - {\bar X} (t) )_H \bigg \} d t \bigg ] \nonumber \\
&& + {\mathbb E} \bigg[ \Phi ( X (T), \mathbb E
[ X(T)])
- \Phi ( {\bar x} (T), \mathbb E [\bar X(T)] )\nonumber
\\&&- \Big( \Phi_x ( {\bar X} (T),\mathbb E
 [\bar X(T)]  )+ \mathbb E[\Phi_{x'} ( {\bar x} (T),\mathbb E
 [\bar X(T)] )], X (T) - {\bar X} (T) \Big)_H \bigg] .
\end{eqnarray}
\end{lem}

\begin{proof}

Suppose that
$(u(\cdot); X(\cdot))$
and $(\bar u(\cdot);\bar X(\cdot))$
are any two given  admissible
control pairs. By the state equation \eqref{eq:3.1}, it is easy to check that
the difference
$X(\cdot)-\bar X(\cdot)$ satisfies  the following MF-SPDE:
\begin{eqnarray}\label{eq:5.10}
\left\{
\begin{aligned}
d (X (t)-\bar X(t)) =& \ [ - A (t) (X(t)
 -\bar X(t))+ h ( s )-\bar h ( s ) ] d t
 +[ g ( s )-\bar g(s) )] d W (t) , \quad t \in [ 0, T ] , \\
X(0)-\bar X(0) =& \ 0
\end{aligned}
\right.
\end{eqnarray}
And by the definition of  the adjoint
  equation (see \eqref{eq:4.1}),
   we can  get that $(\bar p(\cdot), \bar q(\cdot))$ satisfies
 the following MF-BSPDE
\begin{eqnarray} \label{eq:5.10}
\left\{
\begin{aligned}
d \bar p(t) = & - \{ - A^* (t) \bar p (t) + \bar {\cal H}_{x} (t) + {\mathbb E} [\bar  {\cal H}_{x'}]d t + \bar q (t) d W (t) ,  \quad t \in [0, T] , \\
p (T) = & \ \Phi_{x} ( \bar X  (T),
\mathbb E [\bar X(T)] )+\mathbb E[\Phi_{x'} ( \bar X(T),\mathbb E [ X(T)] )] ,
\end{aligned}
\right.
\end{eqnarray}
Then by using It\^o's formula to $( {\bar p} (t), X (t) - {\bar X} (t) )_H$, we  get that
\begin{eqnarray}\label{eq:4.12}
&&{\mathbb E}\bigg [ \int_0^T  \bigg \{ ( {\bar p} (t), h (t) - {\bar h} (t) )_H
+ ( {\bar q} (t), g (t) - {\bar g} (t) )_H \bigg \} d t \bigg ] \nonumber \\
~~~~~~~~~&=& {\mathbb E} \bigg [ \int_0^T ( {\bar {\cal H}}_x (t) + {\mathbb E} [ {\bar {\cal H}}_{x'}(t)]
, X (t) - {\bar X} (t) )_H d t \bigg ]
\nonumber
\\&&+ {\mathbb E} \Big[ (\Phi_{x} ( \bar X  (T),
\mathbb E [\bar X(T)] )+\mathbb E\Phi_{x'} ( \bar X(T),\mathbb E [\bar X(T)] ), X (T) - {\bar X} (T) )_H \Big] .
\end{eqnarray}

 In view of the cost functional  and the definitions of the Hamiltonian  ${\cal H}$ and (see \eqref{eq4.2} and \eqref{eq:5.2}),
we can see that
\begin{eqnarray}\label{eq:4.10}
J ( u (\cdot) ) - J ( {\bar u} (\cdot) )
&=& {\mathbb E} \bigg [ \int_0^T \bigg \{ {\cal H} (t) - {\bar {\cal H}} (t)
- ( {\bar p} (t), h (t) - {\bar h} (t) )_H \nonumber - ( {\bar q} (t), g (t) - {\bar g} (t) )_H \bigg \} d t \bigg ]
 \\&&+ {\mathbb E} \bigg[ \Phi ( X (T), \mathbb E[X(T)])
- \Phi ( {\bar X} (T), \mathbb E [\bar X(T)] )\bigg] .
\end{eqnarray}
Then \eqref{eq:4.9} can be immediately obtained by
  substituting \eqref{eq:4.12} into \eqref{eq:4.10}. The proof is complete.
\end{proof}

Next we derive a variational formula for the cost functional \eqref{eq:5.2}.

\begin{lem}\label{lem:3.1}
 Let Assumption \ref{ass:2.5} be satisfied. Then
 we have the following variational formula
\begin{eqnarray}\label{eq:4.4}
\frac{d}{d\epsilon} J ( {\bar u} (\cdot) + \epsilon ( v (\cdot) - {\bar u} (\cdot) ) ) |_{\epsilon = 0}
&=& \lim_{\epsilon \rightarrow 0^+} \frac{J ( {\bar u} (\cdot) + \epsilon ( v (\cdot) - {\bar u} (\cdot) ) )
-J ( {\bar u} (\cdot) )}{\epsilon} \nonumber \\
&=& {\mathbb E} \bigg [ \int_0^T ( {\bar {\cal H}}_u (t), v (t) - {\bar u} (t) )_U d t \bigg ] .
\end{eqnarray}
where $\bar u(\cdot)$ and $v(\cdot)$
are  any  two given admissible controls, and $0\leq\eps \leq 1$ .
\end{lem}

\begin{proof}

Suppose that $( {\bar u} ( \cdot ); {\bar X} ( \cdot) )$
is a given admissible pair and $( {\bar p} ( \cdot ), {\bar q} ( \cdot ) )$ is the corresponding adjoint process.
Define  a perturbed
control process of $\bar u(\cdot)$ as follows:
\begin{eqnarray} \label{eq:5.15}
u^\epsilon ( \cdot ) \triangleq {\bar u} ( \cdot ) + \epsilon ( v ( \cdot ) - {\bar u} ( \cdot) ) ,
\quad 0 \leq \epsilon \leq 1 ,
\end{eqnarray}
where $v(\cdot)$ is any given
admissible control.
  Due to the convexity of the control domain ${\mathscr U}$,  $u^\eps(\cdot)$ belongs to ${\cal A}$. Let  $X^\epsilon ( \cdot )$ be  the state process corresponding to the control $u^\eps(\cdot)$.  We will use the following
shorthand notation:
\begin{eqnarray} \label{eq:5.16}
{\cal H}^\epsilon (t) \triangleq {\cal H} ( t, X^\epsilon (t), \mathbb EX^\epsilon ( t), u^\epsilon (t),
{\bar p} (t), {\bar q} (t) ) .
\end{eqnarray}
Using  the shorthand notations
\eqref{eq:5.8} and  \eqref{eq:5.16},
 from Lemma \ref{lem:4.3}, we get that

\begin{eqnarray}\label{eq:5.17}
J ( u^{\eps} (\cdot) ) - J ( {\bar u} (\cdot) ) \nonumber
&=& {\mathbb E} \bigg [ \int_0^T \bigg \{ {\cal H}^\eps (t) - {\bar {\cal H}} (t)
- ( {\bar {\cal H}}_x (t) +
 \mathbb E [{\bar {\cal H}}_{x'} (t)],
X (t) - {\bar X} (t) )_H
- ( {\bar {\cal H}}_u (t),
u^\epsilon (t) - {\bar u} (t) )_U \bigg \} d t \bigg]\nonumber
\\&&+ {\mathbb E} \bigg[ \Phi ( X^\eps (T), \mathbb E[X^\eps(T))]
- \Phi ( {\bar X} (T), \mathbb E [\bar X(T)] ) \nonumber
\\&&- ( \Phi_x ( {\bar X} (T),\mathbb E
 [\bar X(T)]  )+ \mathbb E[\Phi_{x'} ( {\bar X} (T),\mathbb E
 [\bar X(T)] )], X^\eps (T) - {\bar X} (T) )_H \bigg] .
 \nonumber \\
&& + {\mathbb E} \bigg [ \int_0^T ( {\bar {\cal H}}_u (t),  u^\epsilon (t) - {\bar u} (t) )_U d t \bigg ] .
\end{eqnarray}
In view of  the Taylor series expansion,
it follows that
\begin{eqnarray}\label{eq:4.121}
{\mathbb E} \bigg [ \int_0^T \{ {\cal H}^\epsilon (t) - {\bar {\cal H}} (t) \} d t \bigg ]
&=& {\mathbb E} \bigg [ \int_0^T \int_0^1 \bigg \{ ( {\cal H}_x^{\epsilon, \lambda} (t), X^\epsilon (t) - {\bar X} (t) )_H
+ ( {\cal H}_{x'}^{\epsilon, \lambda} (t),
 \mathbb E [X^\epsilon ( t)] -\mathbb E [ {\bar X} ( t)] )_H \nonumber \\
&& + ( {\cal H}_u^{\epsilon, \lambda} (t), u^\epsilon (t) - {\bar u} (t) )_U \bigg \}
d \lambda d t \bigg ] \nonumber \\
&=& {\mathbb E} \bigg [ \int_0^T \int_0^1 \bigg \{ ( {\cal H}_x^{\epsilon, \lambda} (t)
+
\mathbb E[{\cal H}_{x'}^{\epsilon, \lambda} (t)],
X^\epsilon (t) - {\bar X} (t) )_H \nonumber \\
&& + ( {\cal H}_u^{\epsilon, \lambda} (t),
u^\epsilon (t) - {\bar u} (t) )_U \bigg \} d \lambda d t \bigg ] ,
\end{eqnarray}
where
\begin{eqnarray*}
{\cal H}^{\epsilon, \lambda} (t) \triangleq {\cal H} ( t, X^{\epsilon, \lambda} (t),
\mathbb E [X^{\epsilon, \lambda} ( t)],
u^{\epsilon, \lambda} (t), {\bar p} (t), {\bar q} (t) ) ,
\end{eqnarray*}
and
\begin{eqnarray*}
X^{\epsilon, \lambda} (t) \triangleq {\bar X} (t) + \lambda ( X^\epsilon (t) - {\bar X} (t) ) , \\
u^{\epsilon, \lambda} (t) \triangleq {\bar u} (t) + \lambda ( u^\epsilon (t) - {\bar u} (t) ) .
\end{eqnarray*}
On the hand,
it follows from  the definition of $u^{\eps}$ (see \eqref{eq:5.15})
that
\begin{eqnarray} \label{eq:5.20}
  \mathbb E
  \bigg[\int_0^T||u^\eps(t)-\bar u(t)||^2_Udt\bigg]=\eps^2
  \mathbb E\bigg[\int_0^T||v(t)-\bar u(t)||^2_Udt\bigg]
\end{eqnarray}
Further, in view of  the continuous dependence theorem of MF-SPDE (see Theorem
\ref{thm:3.4} ), we have
\begin{eqnarray} \label{eq:5.21}
&&{\mathbb E} \bigg [ \sup_{0 \leq t \leq T} \| X^\epsilon (t) - {\bar X} (t) \|^2_H \bigg ]
+ {\mathbb E} \bigg [ \int_0^T \| X^\epsilon (t) - {\bar X} (t) \|^2_V d t \bigg ]\nonumber
 \\&& \leq
 K  \mathbb E\bigg[\int_0^T||u^\eps(t)-\bar u(t)||^2_Udt\bigg]\nonumber
 \\&& =K \eps^2
 \mathbb E \bigg[ \int_0^T||v(t)-\bar u(t)||^2_Udt
 \bigg] \ .
\end{eqnarray}
 Therefore, combining \eqref{eq:4.121},
 \eqref{eq:5.20} and  \eqref{eq:5.21}
  yields
\begin{eqnarray}\label{eq:4.14}
&& {\mathbb E} \bigg [ \int_0^T \bigg \{ {\cal H}^\epsilon (t) - {\bar {\cal H}} (t) - ( {\bar {\cal H}}_x (t) +
\mathbb E{\bar {\cal H}}_{x'}(t), x^\epsilon (t) - {\bar x} (t) )_H
- ( {\bar {\cal H}}_u (t),
u^\epsilon (t) - {\bar u} (t) )_U \bigg \} d t \bigg ] \nonumber \\
&=& {\mathbb E} \bigg [ \int_0^T \int_0^1 \bigg \{ ( {\cal H}_x^{\epsilon, \lambda} (t)+
 \mathbb E[{\cal H}_{x'}^{\epsilon, \lambda} (t)]-{\bar {\cal H}}_x (t) -
\mathbb E[{\bar {\cal H}}_{x'}(t)], X^\epsilon (t) - {\bar X} (t) )_H\nonumber \\
&& + ( {\cal H}_u^{\epsilon, \lambda} (t)-{\bar {\cal H}}_u (t), u^\epsilon (t) - {\bar u} (t) )_U \bigg \}
d \lambda d t \bigg ] \nonumber \\
&\leq& \bigg\{{\mathbb E} \bigg [ \int_0^T \int_0^1
 || ( {\cal H}_x^{\epsilon, \lambda} (t)+
 \mathbb E[{\cal H}_{x'}^{\epsilon, \lambda} (t)]-{\bar {\cal H}}_x (t) -\mathbb E[{\bar {\cal H}}_{x'}(t)]||_H^2 dtd\lambda \bigg]\bigg\}^{\frac{1}{2}}
\bigg\{{\mathbb E} \bigg [ \int_0^T || X^\epsilon (t) - {\bar X} (t) ||_H^2 \bigg]\bigg\}^{\frac{1}{2}}\nonumber \\
&& +\bigg\{{\mathbb E} \bigg [ \int_0^T \int_0^1
 || ( {\cal H}_u^{\epsilon, \lambda} (t)-{\bar {\cal H}}_u (t)||_H^2 dtd\lambda \bigg]\bigg\}^{\frac{1}{2}}
\bigg\{{\mathbb E} \bigg [ \int_0^T || u^\epsilon (t) - {\bar u} (t) ||_H^2 \bigg]\bigg\}^{\frac{1}{2}}\\ \nonumber
&\leq&  K \eps \bigg\{{\mathbb E} \bigg [ \int_0^T \int_0^1
 || ( {\cal H}_x^{\epsilon, \lambda} (t)+
 \mathbb E[{\cal H}_{x'}^{\epsilon, \lambda} (t)]-{\bar {\cal H}}_x (t) -\mathbb E[{\bar {\cal H}}_{x'}(t)]||_H^2 dtd\lambda \bigg]\bigg\}^{\frac{1}{2}}\nonumber \\
&& + K\eps\bigg\{{\mathbb E} \bigg [ \int_0^T \int_0^1
 || ( {\cal H}_u^{\epsilon, \lambda} (t)-{\bar {\cal H}}_u (t)||_H^2 dtd\lambda \nonumber \bigg]\bigg\}^{\frac{1}{2}}\\ \nonumber
 &=& o(\eps)
\end{eqnarray}
where the last equality
can be obtained by
the fact  that
\begin{eqnarray}
 &&\lim_{\eps\longrightarrow 0} \bigg\{{\mathbb E} \bigg [ \int_0^T \int_0^1
 || ( {\cal H}_x^{\epsilon, \lambda} (t)+
 \mathbb E[{\cal H}_{x'}^{\epsilon, \lambda} (t)]-{\bar {\cal H}}_x (t) -\mathbb E[{\bar {\cal H}}_{x'}(t)]||_H^2 dtd\lambda \bigg]\bigg\}^{\frac{1}{2}}\nonumber
 \\&&~~~~+ \lim_{\eps\longrightarrow 0}\bigg\{{\mathbb E} \bigg [ \int_0^T \int_0^1
 || ( {\cal H}_u^{\epsilon, \lambda} (t)-{\bar {\cal H}}_u (t)||_H^2 dtd\lambda \bigg]\bigg\}^{\frac{1}{2}}=0.
\end{eqnarray}
which can be got
by combining  Assumption \ref{ass:2.5}, \eqref{eq:5.20},
 \eqref{eq:5.21} and
 the dominated
convergence theorem.

We can similarly get that
\begin{eqnarray}\label{eq:4.16}
&&{\mathbb E} \bigg[ \Phi ( X^\eps (T), \mathbb E[X^\eps(T)])
- \Phi ( {\bar X} (T), \mathbb E [\bar X(T)] )\nonumber\\&&~~~~- \Big( \Phi_x ( {\bar X} (T),\mathbb E
 [\bar X(T)]  )+ \mathbb E[\Phi_{x'} ( {\bar X} (T),\mathbb E
 [\bar X(T)] )], X^\eps (T) - {\bar X} (T) \Big)_H \bigg] = o(\epsilon) .
\end{eqnarray}
Hence, by substituting  \eqref{eq:4.14}
and \eqref{eq:4.16} into \eqref{eq:5.17}, we get that
\begin{eqnarray*}
\lim_{\epsilon \rightarrow 0} \frac{J ( u^\epsilon (\cdot) ) - J ( {\bar u} (\cdot) )}{\epsilon}
= {\mathbb E} \bigg [ \int_0^T ( {\bar {\cal H}}_u (t), v (t) - {\bar u} (t) )_U d t \bigg ] .
\end{eqnarray*}
The proof is complete.
\end{proof}

\subsection{ Stochastic  Maximum Principle}

In this subsection, we will establish  the necessary and sufficient maximum principle for the optimal control of Problem \ref{pro:2.1}.

\begin{thm}[{\bf Necessary stochastic maximum principle}]
\label{thm:4.2} Let Assumption \ref{ass:2.5} be satisfied.
Let $( {\bar u} (\cdot); {\bar x} (\cdot) )$ be an optimal pair of
Problem \ref{pro:2.1} associated with the adjoint process $( {\bar p} (\cdot), {\bar
q} (\cdot) )$. Then the following minimum condition holds:
\begin{eqnarray}\label{eq4.15}
( {\bar {\cal H}}_u (t) , v - {\bar u} (t) )_U \geq 0,
\end{eqnarray}
$\forall v \in {\mathscr U}$, for a.e. $t \in [ 0, T ]$, ${\mathbb P}$-a.s..
\end{thm}

\begin{proof}
For any admissible control $v (\cdot) \in {\cal A}$, it follows from Lemma \ref{lem:3.1}
that
\begin{eqnarray}
&& {\mathbb E} \bigg [ \int_0^T ( {\bar {\cal H}}_u (t) , v (t) - {\bar u} (t) )_U d t \bigg ]
 \nonumber \\
&& = \lim_{\epsilon \rightarrow 0} \frac{J ( u^\epsilon (\cdot) ) - J ( {\bar u} (\cdot) )}{\epsilon} \geq 0 ,
\end{eqnarray}
where the last inequality can
be get directly since $( {\bar u} (\cdot); {\bar X} (\cdot) )$ is an optimal pair of Problem \ref{pro:2.1} Then  minimum condition
\eqref{eq4.15} can be obtained by the
 classic argument following Bensoussan (1982). For similar proofs, we refer to Meng and Shen (2015). The proof is complete.
\end{proof}

Next we will  give the verification theorem of optimality, namely, the sufficient maximum principle
for the optimal control of Problem \ref{pro:2.1}. Besides Assumption \ref{ass:2.5}, the
verification theorem relies on some convexity assumptions of the Hamiltonian and the terminal cost.

\begin{thm}[{\bf Sufficient maximum principle}]\label{thm:4.1}
Let Assumption \ref{ass:2.5} be satisfied. Let $( {\bar u} (\cdot);  {\bar X} (\cdot) )$ be an
admissible pair associated with the adjoint
 process $( {\bar p} (\cdot), {\bar q} (\cdot) ).$
 Suppose that for almost all $( t, \omega ) \in [ 0, T ] \times \Omega,$
\begin{enumerate}
\item ${\cal H} ( t, x, x', u, {\bar p} (t), {\bar q} (t) )$ is convex in $( x, x', u )$,
\item $\Phi (x,x')$ is convex in $(x,x')$,
\item $
{\bar {\cal H}} (t) = \min_{u \in {\mathscr U} } {\cal H} ( t, {\bar X} (t),  \mathbb E [{\bar X} ( t)],
u,  {\bar p} (t), {\bar q} (t) ) ,
$
\end{enumerate} then $( {\bar u} (\cdot); {\bar X} (\cdot) )$
is an optimal pair of Problem \ref{pro:2.1}.
\end{thm}

\begin{proof}
 Given an arbitrary admissible pair.$( u (\cdot); X (\cdot) ). $  By Lemma \ref{lem:4.3}, we get
\begin{eqnarray}\label{eq:5.22}
J ( u (\cdot) ) - J ( {\bar u} (\cdot) )
&=& {\mathbb E} \bigg [ \int_0^T \bigg \{ {\cal H} (t) - {\bar {\cal H}} (t)
- ( {\bar {\cal H}}_x (t) +
 \mathbb E [{\bar {\cal H}}_{x'} (t)],
X (t) - {\bar X} (t) )_H \bigg \} d t \bigg ] \nonumber \\
&& + {\mathbb E} \bigg[ \Phi ( X (T), \mathbb E
[ X(T)])
- \Phi ( {\bar X} (T), \mathbb E [\bar X(T)] )\nonumber
\\&&- \Big( \Phi_x ( {\bar X} (T),\mathbb E
 [\bar X(T)]  )+ \mathbb E[\Phi_{x'} ( {\bar X} (T),\mathbb E
 [\bar X(T)] )], X (T) - {\bar X} (T) \Big)_H \bigg] .
\end{eqnarray}
By the convexity of
${\cal H} ( t, x, x', u, , {\bar p} (t), {\bar q} (t) )$ and $\Phi (x',x),$  in view of  Proposition 1.54
of Ekeland and T\'emam (1976), we have
\begin{eqnarray}\label{eq:5.3}
{\cal H} (t) - {\bar {\cal H}} (t) &\geq& ( {\bar {\cal H}}_x (t), X (t) - {\bar X} (t) )_H
+ ( {\bar {\cal H}}_{x'} (t), \mathbb E [X ( t)] - \mathbb E[{\bar X} ( t)] )_H  \nonumber \\
&& + ( {\bar {\cal H}}_u (t), u (t) - {\bar u} (t) )_U,
\end{eqnarray}
and
\begin{eqnarray}\label{eq:5.4}
 \Phi ( X (T), \mathbb E
 [X(T)])
- \Phi ( {\bar X} (T), \mathbb E [\bar X(T)] )
 &\geq & \Big( \Phi_x ( {\bar X} (T),\mathbb E
 [\bar X(T)]  ), X (T)- {\bar X} (T) \Big)_H\nonumber
 \\&&+\Big(  \Phi_{x'} ( {\bar X} (T),\mathbb E[
 \bar X(T)] ), \mathbb E[X (T)] -\mathbb E[ {\bar X} (T)] \Big)_H  .
\end{eqnarray}
In addition, in view of the convex optimization principle
(see Proposition 2.21 of Ekeland and T\'emam, 1976), the optimality condition 3 implies that for almost all $( t, \omega ) \in
[ 0, T ] \times \Omega$,
\begin{eqnarray}\label{eq:5.5}
( {\bar {\cal H}}_u (t), u (t) - {\bar u} (t) )_U\geq 0 .
\end{eqnarray}
 Substituting  \eqref{eq:5.3}, \eqref{eq:5.4} and \eqref{eq:5.5} into \eqref{eq:5.22} yields
\begin{eqnarray*}
J ( u (\cdot) ) - J( {\bar u} (\cdot) ) \geq 0 .
\end{eqnarray*}
Therefore, since $u(\cdot)$ is arbitrary,  ${\bar u} (\cdot)$ is an optimal control process and $( {\bar u} (\cdot); {\bar X} (\cdot) )$
is an optimal pair. The proof is complete.
\end{proof}

\subsection{Optimality System of Mean-Field Stochastic Partial  Differential Equation  }

Associated with the optimal control ${\bar u} (\cdot),$
consider the following stochastic system

\begin{eqnarray}\label{eq:49}
\left\{
\begin{aligned}
d \bar X (t) =& \ [ - A (t) \bar X(t) + h ( s,
\bar X (s),
\mathbb E [\bar X (s)], u(s) ) ] d t + g ( s,
\bar X (s),\mathbb E [\bar X (s)],u(s) ) d W (t)  ,
 \\d \bar p(t) = & - \bigg[ - A^* (t) \bar p (t) +
h^*_x(t, \bar X(t), \mathbb E [\bar X(t)], u(t))\bar p(t)
+\mathbb E[ h^*_x(t, \bar X(t), \mathbb E [\bar X(t)], u(t))\bar p(t)]
\\&+g^*_x(t, \bar X(t), \mathbb E [\bar X(t)], u(t))\bar q(t)
+\mathbb E[g^*_x(t, \bar X(t), \mathbb E
[\bar X(t)], u(t))\bar q(t)]\bigg]d t
 \\&
 +l_x(t, \bar X(t), \mathbb E \bar X(t), u(t))
+\mathbb E [l_x(t, X(t), \mathbb E X(t), u(t))]\bigg]d t
 + \bar q (t) d W (t) ,  \quad t \in [0, T] ,  \\
\bar X (0) =& \ x , \\
p (T) = & \ \Phi_{x} ( X (T), \mathbb E [X(T)] )
+\mathbb E \Phi_{x'} ( X (T), \mathbb E [X(T)] ), \\
( {\bar {\cal H}}_u (t)& , v - {\bar u} (t) )_U \geq 0 , \quad \forall v \in {\mathscr U} .
\end{aligned}
\right.
\end{eqnarray}
Note that this is a mean-field fully-coupled forward-backward
stochastic partial differential equation
 consisting of  the  state equation \eqref{eq:5.1}
and adjoint equation \eqref{eq:4.1} associated with the optimal control ${\bar u} (\cdot)$
and  with the coupling presented
in the minimum condition of \eqref{eq4.15}.
 The forward-backward equation (\ref{eq:49})
 is referred to as the stochastic Hamiltonian system or
the optimality system of Problem\ref{pro:2.1}.
 The 4-tuple stochastic process $({\bar u} (\cdot), {\bar X} (\cdot),
{\bar p} (\cdot), {\bar q} (\cdot))\in {\cal M}_{\mathscr F}^2 ( 0, T; U)\times{\cal
M}_{\mathscr F}^2 ( 0, T; V )\times{\cal M}_{\mathscr F}^2 ( 0, T; V )
\times {\cal M}_{\mathscr F}^2 ( 0, T; H )$
satisfying the above is called the solution of
 (\ref{eq:49}).  Under Proper
assumptions, we can claim that the
existence of the optimal control to Problem \ref{pro:2.1} is  equivalent to the
solvability of the stochastic Hamiltonian system \eqref{eq:49}.

\begin{cor}\label{cor:4.1}
 Let Assumption \ref{ass:2.5} and Conditions 1-2
in Theorem \ref{thm:4.1} be satisfied. Then
the existence of the optimal control
to Problem \ref{pro:2.1}is equivalent to the existence of a solution to the stochastic Hamiltonian system.
\eqref{eq:49}.
\end{cor}

\begin{proof}

For the sufficient part, suppose that the stochastic Hamiltonian system \eqref{eq:49}   admits an adapted solution
$({\bar u} (\cdot), {\bar X} (\cdot), {\bar p} (\cdot),
{\bar q} (\cdot))\in {\cal M}_{\mathscr F}^2 ( 0, T; U)\times{\cal M}_{\mathscr F}^2
( 0, T; V )\times{\cal M}_{\mathscr F}^2 ( 0, T; V ) \times {\cal
M}_{\mathscr F}^2 ( 0, T; H )$ , then we begin to prove the existence of the  optimal control
of Problem \ref{pro:2.1}.
In fact, from the
minimum condition in the stochastic Hamiltonian system \eqref{eq:49} and the convexity of  ${\cal H} ( t, \mathbb E[{\bar X'} (t)], {\bar X}(t), u,
{\bar p} (t), {\bar q} (t) )$  with  $u$ , we
know that
\begin{eqnarray*}
{\cal H} ( t, {\bar X} (t), \mathbb E{\bar X} ( t ), {\bar u} (t), {\bar p} (t), {\bar q} (t) ) = \min_{ u\in {\mathscr U}}  {\cal H} ( t, {\bar
X} (t), \mathbb E {\bar X} ( t ), u,  {\bar p} (t), {\bar q} (t) ) .
\end{eqnarray*}
Therefore,  in view of
the sufficient stochastic maximum principle (see Theorem \ref{thm:4.1})  we get that $({\bar u} (\cdot); {\bar X} (\cdot))$ is an optimal pair.

For the necessary part, suppose that $({\bar u} (\cdot); {\bar X} (\cdot))$ is an optimal pair
 associated corresponding adjoint process $({\bar p} (\cdot), {\bar q} (\cdot))$,
then    in view of the necessary stochastic maximum principle,  we  get that the stochastic Hamiltonian system
\eqref{eq:49} has an adapted solution
$({\bar u} (\cdot), {\bar X} (\cdot), {\bar p} (\cdot), {\bar q} (\cdot))
\in {\cal M}_{\mathscr F}^2 ( 0, T; U)\times{\cal M}_{\mathscr F}^2 ( 0, T; V
)\times{\cal M}_{\mathscr F}^2 ( 0, T; V ) \times {\cal M}_{\mathscr F}^2 (
0, T; H )$ . The proof is complete.
\end{proof}

\section{An Application: Linear-Quadratic Optimal Control Problems For Mean-Field Stochastic Partial Differential Equation}
The case where the system dynamics are described by a set of linear differential equations and the cost is described by a quadratic function is called the LQ problem which is
one of the most important optimal control problems.  The reader is referred to  Chapter 6 in Yong and Zhou
(1999) for a complete  survey on this topic.
In this section, an infinite-dimensional lQ problem of mean-field type will be discussed. As an application,  we will solve an lQ problem
a Cauchy problem of a stochastic linear parabolic PDE  of mean field type.

\subsection{ LQ Optimal Control  of Mean- Field
Stochastic Partial Differential Equation }
This subsection is devoted to   applying the stochastic maximum principles to study
an infinite-dimensional linear-quadratic optimal control problem of mean field type, and establish the explicit dual characterization of the optimal
control with stochastic Hamiltonian system of mean field type.

  Consider the following  linear
  quadratic optimal control problem minimize
  over ${\cal A} = {\cal M}^2_{\mathscr F} ( 0, T; U )$ the following quadratic
  cost functional

  \begin{eqnarray}\label{eq:6.1}
\begin{split}
J ( u (\cdot) ) = &{\mathbb E} [ ( \Phi_1 X (T), X(T) )_H ]+{\mathbb E} [ ( \Phi_2  \mathbb E[X (T)],
 \mathbb E[X (T)] )_H ]
\\&+{\mathbb E} \bigg [ \int_0^T ( G_1 (s) X (s), X
(s) )_H d s \bigg ]+ {\mathbb E} \bigg [ \int_0^T ( G_2 (s)  \mathbb E[X (s)],\mathbb E [X(s)] )_H d s \bigg ]
\\&+ {\mathbb E} \bigg [ \int_0^T ( N (s) u (s), u (s) )_U d s
\bigg ],
\end{split}
\end{eqnarray}
  where  $X(\cdot)$ is  the
  solution of the
 controlled
linear MF-SPDE in the Gelfand triple $(V, H, V^*)$:
\begin{eqnarray}\label{eq:6.2}
\left\{
\begin{aligned}
d X (t) =& \ [ - A (t) X (t)
+ B_1(t ) X( t)+ B_2 (t ) \mathbb E
[ X ( t)]
+ C (t) u (t) ] d t+ [ D_1 (t) X (t) + D_2 ( t )
\mathbb E [X ( t  )]
+ F (t) u (t) ] d W (t) ,  \\
X(0) =& \ x  , \quad t \in [0, T] ,
\end{aligned}
\right.
\end{eqnarray}
Here $ A, B_1, B_2, C, D_1, D_2, F,  G_1, G_2, N,\Phi_1$ and $\Phi_2$ are given random mappings such that $A: [ 0, T ] \times \Omega \rightarrow {\mathscr L}
( V, V^* )$, $B_1, B_2,  D_1, D_2, G_1, G_2: [ 0, T ] \times \Omega \rightarrow
{\mathscr L} ( H, H )$, $C,F: [ 0, T ] \times \Omega
\rightarrow {\mathscr L} ( U, H )$, $N: [ 0, T ] \times \Omega
\rightarrow {\mathscr L} ( U, U )$ and $\Phi_1, \Phi_2: \Omega \rightarrow
{\mathscr L} ( H, H )$, satisfying the following assumptions:

\begin{ass}\label{ass:5.1} The operator $A$ satisfies
the coercivity and boundedness conditions, i.e., (i) and (ii) in
Assumption \ref{ass:1.1}.  The mappings  $A, B_1, B_2, C, D_1, D_2, F,  G_1, G_2, N,$ $G_1, G_2$ and $N$ are uniformly bounded ${\mathbb
F}$-predictable processes, $\Phi_1$ and $\Phi_2$ is a uniformly bounded ${\cal
F}_T$-measurable random variable.
\end{ass}

\begin{ass}\label{ass:5.2}
The stochastic processes $G_1, G_2$, $N$ and the random variables $\Phi_1$ and $ \Phi_2$  are
nonnegative operators, a.e. $t \in [ 0, T ]$, ${\mathbb P}$-a.s..
Moreover, $N$ is uniformly positive a.e. $t \in [ 0, T ]$, ${\mathbb
P}$-a.s., i.e., for $\forall u \in U$, $( N u, u )_U \geq k ( u, u
)_U$, for some positive constant $k$, a.e. $t \in [ 0, T ]$,
${\mathbb P}$-a.s..
\end{ass}
 In the general control problem \ref{pro:2.1}, we
 specify the coefficients $h$, $g$, $l$ and $\Phi$
 with
\begin{eqnarray*}
&& h ( t, x, x', u) =  B_1(t ) x + B_2 (t )x'
+ C (t) u  , \\
&& g ( t, x, x', u ) = D_1 (t) x + D_2 ( t ) x'
+ F (t) u , \\
&& l ( t, x, x',  u ) = ( G_1 (t) x,  x )_H
+( G_2 (t) x',  x' )_H +( N (t) u, u )_U
 \\&&\Phi (x, x') = ( \Phi_1 x,  x )_H+( \Phi_2 x',  x' )_H
.
\end{eqnarray*}
 By Assumptions \ref{ass:5.1} and \ref{ass:5.2},
  it is easily to check that Assumption \ref{ass:2.5} on the coefficients
$(A, h, g, l, \Phi)$ holds. So our LQ Problem
 can be embedded in Problem\ref{pro:2.1}.
In this case, the  Hamiltonian ${\cal H}$
has the following form:
\begin{eqnarray}\label{eq:6.7}
{\cal H} ( t, x, x', u, p, q )
&=& (B_1(t)x+ B _2(t) x' + C (t) u, p )_H  + ( D_1 (t) x + D_2 ( t ) x + F (t) u, q )_H
+ ( G _1(t) x, x )_H \\&&+ ( G _2(t) x, x )_H + ( N (t) u, u )_U. \nonumber
\end{eqnarray}
 Here  we  denote the adjoint operators of $B_1, B_2$, $C_1$, $C_2$, $D$, and $F$ by  $B^*_1, B_2^*$, $C^*_1$, $C^*_2$, $D^*$ and $F^*_1$, respectively.
Associated with an  admissible
pair $( u (\cdot); X (\cdot) ),$ the adjoint equation \eqref{eq:4.1}  has the following form:
\begin{eqnarray}\label{eq:7.6}
\left\{
\begin{aligned}
d p(t) =& - \{- A^*(t) p (t)
 +B^*_1(t)X(t)+
 + \mathbb E[B^*_2 ( t) p (t)] + D_1^* (t) q (t) + \mathbb E[ D_2^* (t)q_2(t)]
 \\&+ 2 G_1 (t) X (t)
 +2 \mathbb E[ G_2 (t) X (t)]\} d t + q (t) d W (t) ,  \quad t \in [0, T] , \\
p (T) =& \ 2 \Phi_1 X (T)+ 2 \mathbb E [\Phi_2 X (T)].
\end{aligned}
\right.
\end{eqnarray}
 Because  in this case there is no constraint on the control, the minimum condition \eqref{eq4.15} of the optimal control is
\begin{eqnarray}\label{eq:6.5}
{\cal H}_u ( t, X (t), \mathbb E [X ( t)], p (t), q (t), u (t) ) = 0 .
\end{eqnarray}
Therefore the stochastic Hamiltonian system is  the following
fully-coupled linear forward-backward stochastic partial differential  equation
\begin{eqnarray}\label{eq:6.8}
\left\{
\begin{aligned}
d X (t) =& \ [ - A (t) x (t)
+ B_1(t ) X ( t)+ B_2 (t ) \mathbb E [X ( t)]
+ C (t) u (t) ] d t+ [ D_1 (t) X (t) + D_2 ( t )
\mathbb E [X ( t  )]
+ F (t) u (t) ] d W (t)  , \\
d p(t) =& - \{- A^*(t) p (t)
 +B^*_1(t)X(t)+
  \mathbb E[B^*_2 ( t) p (t)] + D_1^* (t) q (t) + \mathbb E[ D_2^* (t)q_2(t)]
 \\&+ 2 G_1 (t) X (t)
 +2 \mathbb E[ G_2 (t) X (t)]\} d t + q (t) d W (t) , \quad t \in [0, T] , \\
x (0) =& \ x ,\\
p (T) =& \ 2 \Phi_1 X (T)+2 \mathbb E[ \Phi_2X(T)] , \\
{\cal H}_u ( t&, X (t), \mathbb E [X ( t)], p (t), q (t), u (t) )  = 0 . \\
\end{aligned}
\right.
\end{eqnarray}
Now we give the dual
characterization of the optimal control.

\begin{thm}\label{them:4.1}
 Let Assumptions \ref{ass:5.1}-\ref{ass:5.2}
 be satisfied.  Then our LQ problem has a
 unique optimal control which implies that
 the stochastic Hamiltonian system \eqref{eq:6.8} has a unique adapted solution
$( {\bar u} (\cdot), {\bar X} (\cdot), {\bar p} (\cdot), {\bar q} (\cdot))\in {\cal M}_{\mathscr F}^2 ( 0, T; U)
\times {\cal M}_{\mathscr F}^2 ( 0, T; V )\times{\cal M}_{\mathscr F}^2 ( 0, T; V ) \times {\cal M}_{\mathscr F}^2 ( 0, T; H ).$
Moreover the optimal control is given by
\begin{eqnarray}\label{eq:6.71}
{\bar u} (t) &=& - \frac{1}{2} N ^{-1}(t) \big
 [C^* (t) {\bar p} (t) + F^* (t) {\bar q} (t)
 \big ] .
\end{eqnarray}
\end{thm}

\begin{proof}
 In view of the continuous dependence theorem of MF-SPDE ( see
Theorem \ref{lem:1.4}),  we have that  the cost functional
$J ( u (\cdot) )$ is continuous over
${\cal M}^2_{\mathscr F} ( 0, T; U )$.  From
the uniformly strictly positivity of the process $N,$ we conclude that the cost functional $J
( u (\cdot) )$ is strictly convex and
\begin{eqnarray*}
J ( u (\cdot) ) \geq k {\mathbb E} \bigg [ \int^T_0 || u (t) ||^2 _Ud t \bigg ]
= k \| u (\cdot) \|^2_{{\cal M}^2_{\mathscr F} ( 0, T; U )} .
\end{eqnarray*}
 Therefore the cost functional $J ( u (\cdot) )$ is coercive, i.e.,
\begin{eqnarray*}
\lim_{ \| u (\cdot) \|_{{\cal M}^2_{\mathscr F} ( 0, T; U )} {\rightarrow \infty} } J ( u (\cdot) ) = \infty .
\end{eqnarray*}
 In the end, we get  the uniqueness and existence of the optimal control
${\bar u} (\cdot) \in {\cal M}^2_{\mathscr F} ( 0, T; U )$ of our LQ problem by Proposition 2.12 of Ekeland and T\'emam (1976).

Now we begin to prove that the stochastic Hamiltonian system
\eqref{eq:6.8} has a unique adapted solution. Indeed in view of  Corollary \ref{cor:4.1},
the existence of the optimal control ${\bar u} (\cdot)$ of our LQ problem \ref{pro:2.1} implies that the
stochastic Hamiltonian system \eqref{eq:6.8} has a solution $( {\bar u} (\cdot), {\bar x} (\cdot), {\bar p} (\cdot), {\bar q} (\cdot))
\in {\cal M}_{\mathscr F}^2 ( 0, T; U)\times{\cal M}_{\mathscr F}^2 ( 0, T; V )\times{\cal M}_{\mathscr F}^2 ( 0, T; V )
\times {\cal M}_{\mathscr F}^2 ( 0, T; H )$,
Here ${\bar x} (\cdot)$ is the optimal state and $({\bar p} (\cdot), {\bar q} (\cdot))$ is
the adjoint process corresponding the optimal control ${\bar u} (\cdot)$.
If the stochastic Hamiltonian system \eqref{eq:6.8} has another adapted solution $({\bar u}^\prime (\cdot),
{\bar x}^\prime (\cdot), {\bar p}^\prime (\cdot), {\bar q}^\prime (\cdot) )$, then  view of  Corollary \ref{cor:4.1},
$({\bar u}^\prime (\cdot); {\bar X}^\prime (\cdot))$ have  to  be an optimal pair of  our  LQ Problem.
So  ${\bar u} (\cdot)= {\bar u}^\prime (\cdot)$
 due to the uniqueness of  the optimal control. Moreover,
  from the uniqueness of solutions to MF-SPDE (see Theorem \ref{lem:1.1}) and MF-BSPDE (see Theorem \ref{lem:1.3}),
we get  $( {\bar x} (\cdot), {\bar p} (\cdot), {\bar q} (\cdot)) = ( {\bar x}^\prime (\cdot), {\bar p}^\prime (\cdot),
{\bar q}^\prime (\cdot))$. Therefore, the stochastic Hamiltonian system \eqref{eq:6.8} admits a unique solution. In the end, the dual characterization
 \eqref{eq:6.71} of the unique optimal
 can  be directly  obtained by solving  the minimum condition  \eqref{eq:6.5}.
\end{proof}

\subsection{LQ control of the Cauchy
problem for stochastic linear PDE of Mean Field Type}

In this subsection,  in terms of the results in the previous subsection,  we solve a LQ problem of   a Cauchy problem for a controlled stochastic linear PDE of mean-field type.

 Now we give  some preliminaries of Sobolev spaces.
For $m = 0, 1$,  introduce  the space $H^m \triangleq \{ \phi:
\partial_z^\alpha \phi \in L^2 ( {\mathbb R}^d ), \ \mbox {for any}
\ \alpha: =( \alpha_1, \cdots, \alpha_d ) \ \mbox {with} \ |\alpha|
:= | \alpha_1 | + \cdots + | \alpha_d | \leq m \}$ with the norm
\begin{eqnarray*}
\| \phi \|_m \triangleq \left \{ \sum_{ |\alpha| \leq m } \int_{{\mathbb R}^d}
| \partial_z^\alpha \phi (z) |^2 d z \right \}^{\frac{1}{2}} .
\end{eqnarray*}
The dual space of $H^1$ is  denoted by $H^{-1}$. Put $V = H^1$, $H
= H^0$, $V^* = H^{-1}$. Then we claim that $( V, H, V^* )$ is   a Gelfand triple.

  Suppose that the control domain is    ${\mathscr U} = U = H$. For any
admissible control $u (\cdot, \cdot) \in {\cal M}^2_{\mathscr F} ( 0, T; U )$,
 we introduce  the  controlled  Cauchy problem, where the state process
is the  following  stochastic partial differential equation of mean-field type in
divergence form:
\begin{eqnarray}\label{eq6.13}
\left\{
\begin{aligned}
d y ( t,z ) = & \ \big \{ \partial_{z^i} [ a^{ij} ( t, z )
\partial_{z^j} y ( t, z ) ]
+ b^i ( t, z ) \partial_{z^i} y ( t, z ) + c ( t, z ) y ( t, z ) + \eta ( t, z )
\mathbb E [ y ( t, z )] + u ( t, z ) \big \} d t \\
& + [ \rho ( t, z ) y ( t, z ) + \sigma ( t, z )
 \mathbb E[y ( t,z)] +
u ( t, z )  ] d W (t) ,
\quad ( t, z ) \in [ 0, T ] \times {\mathbb R}^d , \\
y ( 0, z ) =&x \in  {\mathbb R}^d ,
\end{aligned}
\right.
\end{eqnarray}
and the cost functional is
\begin{eqnarray}\label{eq:6.14}
\begin{split}
&\inf_{u (\cdot, \cdot) \in {\cal M}^2_{\mathscr F} ( 0, T; U )} \bigg \{
{\mathbb E} \bigg [ \int_{{\mathbb R}^d} y^2 ( T, z ) d z+{\mathbb E} \bigg [ \int_{{\mathbb R}^d}
 |{\mathbb E} y( T, z )|^2 d z + \iint_{[
0, T ] \times {{\mathbb R}^d}} y^2 ( s, z ) d z  d s \\&~~~~~~~~~+ \iint_{[
0, T ] \times {{\mathbb R}^d}} |\mathbb E y ( s, z ) |^2 d z  d s
+ \iint_{[ 0, T
] \times {{\mathbb R}^d}} u^2 ( s, z ) d z d s \bigg ] \bigg \} .
\end{split}
\end{eqnarray}
Here the coefficients $a^{ij}$, $b^i$, $c$, $\eta$, $\rho$, $\sigma$ and
the initial condition $x$ are given  random functions satisfying
the following assumptions, for some fixed constants $K \in ( 1,
\infty )$ and $\kappa \in ( 0,1 )$:

\begin{ass}\label{ass:6.3}
 $a^{ij}$, $b^i$, $c$, $\eta$, $\rho$ and $\sigma$ are
${\mathscr P} \times {\mathscr B} ({\mathbb R}^d )$-measurable taking
values in the space of real symmetric $d \times d$ matrices, ${\mathbb
R}^{d}$, ${\mathbb R}$, ${\mathbb R}$, ${\mathbb R}$ and ${\mathbb
R}$, respectively, and are bounded by $K$.
\end{ass}

\begin{ass}\label{ass:6.4}
$a^{ij}$ satisfies the following super-parabolic condition
\begin{eqnarray*}
\kappa I \leq 2 a^{ij} ( t, \omega, z ) \leq K I , \quad \forall ( t, \omega, z ) \in
[ 0, T ] \times \Omega \times {\mathbb R}^d ,
\end{eqnarray*}
where $I$ is the $(d \times d)$-identity matrix.
\end{ass}

 In this case, in the Gelfand triple $(V, H, V^*)$ the state equation \eqref{eq6.13}
  can be written as the abstract MF-SPDE
:
\begin{eqnarray}\label{eq:10.1}
\left\{
\begin{aligned}
d y (t) = & \ [ - A (t) y (t) +B_2(t) \mathbb E [y(t)]+ u (t) ] d t + [ D_1 (t) y (t) + D_2 ( t ) \mathbb E [y ( t )] + u (t)] d W (t) ,  \quad t \in [0, T] , \\
y(0) =& x  ,
\end{aligned}
\right.
\end{eqnarray}
where the operators $A$, $B_2$, $D_1$, $D_2$ are
denoted  by
\begin{eqnarray*}
&& A (t) \phi (z) \triangleq - \partial_{z^i} [ a^{ij} ( t, z )
\partial_{z^j} \phi (z) ]
- b^i ( t, z ) \partial_{z^i} \phi (z) - c ( t, z ) \phi (z)  , \quad \forall \phi \in V , \\
&& B_2 (t) \phi (z) \triangleq \eta ( t, z )
, \quad D_1 (t)
\phi (z) \triangleq \rho ( t, z ) \phi (z) , \quad D_2 (t) \phi (z)
\triangleq \sigma ( t , z ) \phi (z) , \quad \forall \phi \in H .
\end{eqnarray*}
Then  we write the optimal control problem as
\begin{eqnarray}\label{eq:6.17}
\begin{split}
&\inf_{u (\cdot) \in {\cal M}^2_{\mathscr F} ( 0, T; U )} \bigg \{
{\mathbb E} \bigg [ ( y (T), y (T))_H + \int_0^T ( y (s), y (s) )_H
d s+ ( \mathbb E[y (T)],
 \mathbb E [y (T)])_H \\&
 ~~~~~~~~~~~~~~+ \int_0^T ( \mathbb E [y (s)],
  \mathbb E [y (s)] )_H
d s + \int_0^T ( u (s), u (s) )_H d s \bigg ] \bigg \} .
\end{split}
\end{eqnarray}\\
Thus this optimal control problem becomes a special case of  our LQ problem in the previous subsection, where $C$, $F$, $N$ and $G$ are
identity operators and $B_1=0$.  From  Assumptions \ref{ass:6.3}-\ref{ass:6.4}, it is easy to check that the optimal control problem \eqref{eq:6.17} satisfies
Assumptions \ref{ass:5.1}-\ref{ass:5.2}. So in view of  Theorem
\ref{them:4.1}, we claim that the optimal control ${\bar u}
(\cdot)$  has the following explicit characterization:
\begin{eqnarray*}
{\bar u} (t) = -\frac{1}{2} \big \{ {\bar p} (t) + {\bar q} (t)\big \} ,
\end{eqnarray*}
where  $( {\bar p} (\cdot), {\bar q} (\cdot) )$ is the unique
solution of the adjoint equation
\begin{eqnarray}\label{eq:6.19}
\left\{
\begin{aligned}
d p (t) =& - \big \{ - A^*(t) p (t) + {\mathbb E} [ B^*_2 ( t) p ( t)]
+ D_1^* (t) q (t) \\
& + {\mathbb E} [ D_2^* ( t  ) q ( t)]+ 2 y (t) \big \} d t
+ q (t) d W (t) , \quad t \in [0, T] , \\
p (T) =& \ 2 y (T) .
\end{aligned}
\right.
\end{eqnarray}
Here $A^*$, $B^*_2$, $D_1^*$, $D_2^*$ denote the adjoint operators of $A$, $B$,
$D_1$, $D_2$. More specifically,
\begin{eqnarray*}
A^*(t) \phi (z)= - \partial_{z^i} [ a^{ij}
( t, z ) \partial_{z^j} \phi (z) ] + b^i ( t, z ) \partial_{z^i}
\phi (z) - [c ( t, z )-\partial_{z^i}b^i(t,z)] \phi (z) , \quad \forall \phi \in V ,
\end{eqnarray*}
and $B^*_2 = B_2$, $D_1^* = D_1$, $D_2^* = D_2$.

\section{Conclusion}

In  this paper we studied the MF-SPDE
and MF-BSPDE and the corresponding
optimal control problem for MF-SPDE.
We establish the  basic
results on  the continuous dependence
theory and existence and uniqueness
theory of
the solution to MF-SPDE
and MF-BSPDE under the proper assumptions on
the coefficients, respectively.
For the optimal control problem of MF-SPDE, we
obtain the necessary and sufficient optimality
conditions for the existence of the
optimal control by convex
variation techniques and duality theory.
An an application, the LQ problem for  MF-SPDE
is investigated to illustrate our optimal control theory result established. As a result, the existence,  uniqueness and  explicit duality presentation of
the optimal control are obtained.

\end{document}